\documentclass[11pt]{article}

\usepackage{algorithm} 
\usepackage{algpseudocode}  
\usepackage{amsmath}  
\usepackage[colorlinks=true,citecolor=blue]{hyperref}
\usepackage{amsfonts}
\usepackage{amssymb,amsmath}
\usepackage{graphicx}
\usepackage{cite}
\usepackage{color}
\usepackage{amsthm}
\usepackage{wrapfig}
\usepackage[figuresright]{rotating}
 \usepackage{subfig}

\usepackage{float}
\numberwithin{equation}{section}

\newtheorem{theorem}{Theorem}[section]
\newtheorem{definition}{Definition}[section]
\newtheorem{lemma}{Lemma}[section]

\newtheorem{remark}{Remark}[section]
\newtheorem{Assumption}{Assumption}[section]
\allowdisplaybreaks

\begin{document}
\vspace{1.3cm}
\title
{Two-step inertial Bregman alternating structure-adapted proximal gradient descent algorithm for nonconvex and nonsmooth problems{\thanks{Supported by Scientific Research Project of Tianjin Municipal Education Commission (2022ZD007).}}}
\author{ {\sc Chenzheng Guo{\thanks{Email: g13526199036@163.com}},
Jing Zhao{\thanks{Corresponding author. Email: zhaojing200103@163.com}}}\\
\small College of Science, Civil Aviation University of China, Tianjin 300300, China\\
}
\date{}
\date{}
\maketitle{}
{\bf Abstract.}
In this paper, we propose accelerated alternating structure-adapted proximal gradient descent method for a class of nonconvex and nonsmooth nonseparable problems. The proposed  algorithm is a monotone method which combines two-step inertial extrapolation and Bregman distance. Under some assumptions, we prove that every cluster point of the sequence generated by the proposed algorithm is a critical point. Furthermore, with the help of Kurdyka--{\L}ojasiewicz property, we establish the convergence of the whole sequence generated by proposed algorithm.
In order to make the algorithm more effective and flexible, we also use some strategies to update the extrapolation parameter and solve the problems with unknown Lipschitz constant. Moreover, we report some preliminary numerical results on involving nonconvex quadratic programming and sparse logistic regression to show the feasibility and effectiveness of the proposed methods.
\vskip 0.4 true cm
\noindent {\bf Key words}: Accelerated methods; Nonconvex and nonsmooth nonseparable optimization; Extrapolation; Bregman distance; Kurdyka--{\L}ojasiewicz property.
\section{Introduction}
\hspace*{\parindent} In this paper, we will consider to solve the following nonconvex and nonsmooth nonseparable optimization problem:
\begin{equation}
\label{MP}
\min_{x\in \mathbb R^{n} ,y\in\mathbb R^{m}}   L(x,y)=f(x)+Q(x,y)+g(y),
\end{equation}
where $f:\mathbb R^{n}\to \mathbb R$, $g:\mathbb R^{m}\to \mathbb R$ are  continuously differentiable, $Q:\mathbb R^{n}\times \mathbb R^{m}\to \mathbb R\cup \left \{ \infty  \right \} $ is a proper, lower semicontinuous function. 
Note that here and throughout the paper, no convexity is assumed on the objective function. Problem \eqref{MP} is used in many application scenarios, such as signal recovery \cite{NNZC,GZZF,B}, nonnegative matrix facorization \cite{PT,LS}, quadratic fractional programming \cite{BCV,XX}, compressed sensing \cite{ABSS,DC},  sparse logistic regression \cite{XX} and so on.
\par The natural method to solve problem \eqref{MP} is the alternating minimization (AM) method (also called block coordinate descent (BCD) method), which, from a given initial point $\left ( x_{0},y_{0}   \right ) \in \mathbb{R}^n  \times \mathbb{R}^m $, generates the iterative sequence $\left \{ \left ( x_{k},y_{k}   \right )  \right \} $ via the scheme:
\begin{equation}
\label{TTB}
\begin{cases}
x_{k+1}\in \arg\min_{ x\in \mathbb{R}^n}\{L(x,y_k)\},\\
y_{k+1}\in \arg\min_{y\in \mathbb{R}^m}\{L(x_{k+1},y)\}.\\
\end{cases}
\end{equation}
If $L(x,y)$ is convex and continuously differentiable, and it is strict convex of one argument while the other is fixed, then the sequence converges to a critical point \cite{BN,BTL}.
\par To relax the requirements of AM method and remove the strict convexity assumption, Auslender \cite{A} introduced proximal terms to \eqref{TTB} for convex function $L$:
\begin{equation}
\label{PAMA}
\begin{cases}
x_{k+1}\in \arg\min_{ x\in \mathbb{R}^n}\{L(x,y_k)+\frac{1}{2\lambda_k}\|x-x_k\|^2_2\},\\
y_{k+1}\in \arg\min_{y\in \mathbb{R}^m}\{L(x_{k+1},y)+\frac{1}{2\mu_k}\|y-y_k\|^2_2\},\\
\end{cases}
\end{equation}
where $\{\lambda_k\}_{k\in\mathbb{N}}$ and $\{\mu_k\}_{k\in\mathbb{N}}$ are positive sequences. The above proximal point method, which is called proximal alternating minimization (PAM) algorithm, was further extended to nonconvex nonsmooth functions. Such as, in \cite{ABR}, {Attouch} et al. applied \eqref{PAMA} to solve nonconvex problem \eqref{MP} and proved the sequence generated by the proximal alternating minimization algorithm \eqref{PAMA} converges to a critical point. More convergence analysis of the proximal point method can be found in \cite{ARS,XY}.
\par Because the proximal alternating minimization algorithm requires an exact solution at each iteration step, the subproblems are very expensive if the minimizers of subproblems are not given in a closed form. The linearization technique is one of the effective methods to overcome the absence of an analytic solution to the subproblem. {Bolte} et al. \cite{BST} proposed the following proximal alternating linearized minimization (PALM) algorithm under the condition that the coupling term $Q(x,y)$ is continuously differentiable:
\begin{equation}
\label{PALM}
\begin{cases}
x_{k+1}\in \arg\min_{ x\in \mathbb{R}^n}\{f(x)+\langle x,\nabla_xQ(x_k,y_k)\rangle+\frac{1}{2\lambda_k}\|x-x_k\|^2_2\},\\
y_{k+1}\in \arg\min_{y\in \mathbb{R}^m}\{g(y)+\langle y,\nabla_yQ(x_{k+1},y_k)\rangle+\frac{1}{2\mu_k}\|y-y_k\|^2_2\}.\\
\end{cases}
\end{equation}
The step size $\lambda_k$ and $\mu_k$ are limited to 
$$\lambda_k\in (0,({\rm Lip}(\bigtriangledown _{x}Q(\cdot ,y_k) ))^{-1} ), \mu_k\in (0,({\rm Lip}(\bigtriangledown _{y}Q(x_{k+1},\cdot ) ))^{-1} ),$$
where ``Lip” denotes the Lipschitz constant of the function in parentheses. In this way, the solution of some subproblems may be expressed by a closed-form or can be easily calculated. 
The global convergence result was established if $L(x,y)$ satisfied the Kurdyka--{\L}ojasiewicz property. 
\par When $f$ and $g$ are continuously differentiable, a natural idea is to linearize $f$ and $g$. {Nikolova} et al. \cite{NT} proposed the corresponding algorithm, called alternating structure-adapted proximal gradient descent (ASAP) algorithm with the following scheme:
\begin{equation}
\label{ASAP}
\begin{cases}
x_{k+1}\in \arg\underset{x\in \mathbb{R}^n}{\min}\{Q(x,y_k)+\langle x,\nabla f(x_k)\rangle+\frac{1}{2\sigma}\|x-x_k\|^2_2\},\\
y_{k+1}\in \arg\underset{y\in \mathbb{R}^m}{\min}\{Q(x_{k+1},y)+\langle y,\nabla g(y_k)\rangle+\frac{1}{2\tau}\|y-y_k\|^2_2\},\\
\end{cases}
\end{equation}
where $\sigma \in (0,({\rm Lip}(\bigtriangledown f))^{-1} ), \tau \in (0,({\rm Lip}(\bigtriangledown g))^{-1} )$. With the help of Kurdyka--{\L}ojasiewicz property they establish the convergence of the whole sequence generated by \eqref{ASAP}.
\par The inertial extrapolation technique has been widely used to accelerate the iterative algorithms for convex  and nonconvex optimizations, since the cost of each iteration stays basically unchanged \cite{OCB,BC}. The inertial scheme, starting from the so-called heavy ball method of {Polyak} \cite{PB}, was recently proved to be very efficient in accelerating numerical methods, especially the first-order methods. Alvarez et al. \cite{AA} applied the inertial strategy to the proximal point
method and proved that it could improve the rate of convergence. The main  feature of the idea is that the new iteration use the previous two or more iterations. 
\par Based on \eqref{PALM}, {Pock and Sabach \cite{PS}} proposed the following inertial proximal alternating linearized minimization (iPALM) algorithm:  
\begin{equation}
\label{iPALM}
\begin{cases}
u_{1k}=x_k+\alpha _{1k}(x_k-x_{k-1}), v_{1k}=x_k+\beta _{1k}(x_k-x_{k-1}),\\
x_{k+1}\in \arg\min_{ x\in \mathbb{R}^n}\{f(x)+\langle x,\nabla_xQ(v_{1k},y_k)\rangle+\frac{1}{2\lambda_k}\|x-u_{1k}\|^2_2\},\\
u_{2k}=y_k+\alpha _{2k}(y_k-y_{k-1}), v_{2k}=y_k+\beta _{2k}(y_k-y_{k-1}),\\
y_{k+1}\in \arg\min_{y\in \mathbb{R}^m}\{g(y)+\langle y,\nabla_yQ(x_{k+1},v_{2k})\rangle+\frac{1}{2\mu_k}\|y-u_{2k}\|^2_2\},\\
\end{cases}
\end{equation}
where $\alpha _{1k},\alpha _{2k},\beta _{1k},\beta _{2k}\in \left [ 0,1 \right ] $. They proved that the generated sequence globally converges to critical point of the objective function under the condition of the Kurdyka--{\L}ojasiewicz property. When $\alpha _{1k}\equiv \alpha _{2k}\equiv \beta _{1k}\equiv \beta _{2k}\equiv 0$, iPALM reduces to PALM. Then {Cai} et al. \cite{GCH} presented a Gauss--Seidel type inertial proximal alternating linearized minimization (GiPALM) algorithm for solving problem \eqref{MP}:
\begin{equation}
\label{GiPALM}
\begin{cases}
x_{k+1}\in \arg\min_{ x\in \mathbb{R}^n}\{f(x)+\langle x,\nabla_xQ(\tilde{x}_{k},\tilde{y}_k)\rangle+\frac{1}{2\lambda_k}\|x-\tilde{x}_{k}\|^2_2\},\\
\tilde{x}_{k+1}=x_{k+1}+\alpha (x_{k+1}-\tilde{x}_{k}), \alpha \in [0,1),\\
y_{k+1}\in \arg\min_{y\in \mathbb{R}^m}\{g(y)+\langle y,\nabla_yQ(\tilde{x}_{k+1},\tilde{y}_{k})\rangle+\frac{1}{2\mu_k}\|y-\tilde{y}_{k}\|^2_2\},\\
\tilde{y}_{k+1}=y_{k+1}+\beta(y_{k+1}-\tilde{y}_{k}), \beta \in [0,1).
\end{cases}
\end{equation}
\par By using inertial extrapolation technique, {Xu} et al. \cite{XX} proposed the following accelerated alternating structure-adapted proximal gradient descent (aASAP) algorithm:
\begin{equation}
\label{aASAP}
\begin{cases}
x_{k+1}\in \arg\min_{ x\in \mathbb{R}^n}\{Q(x,\hat{y} _k)+\langle \nabla f(\hat{x}_k),x\rangle+\frac{1}{2\sigma}\|x-\hat{x} _k\|^2_2\},\\
y_{k+1}\in \arg\min_{y\in \mathbb{R}^m}\{Q(x_{k+1},y)+\langle \nabla g(\hat{y} _k),y\rangle+\frac{1}{2\tau}\|y-\hat{y}_k\|^2_2\},\\
u_{k+1}=x_{k+1}+\beta_{k}(x_{k+1}-x_{k}), v_{k+1}=y_{k+1}+\beta _{k}(y_{k+1}-y_{k}),\\
{\rm if}  \ L(u_{k+1},v_{k+1})\le L(x_{k+1},y_{k+1}), \ {\rm then} \ \hat{x} _{k+1}=u_{k+1}, \hat{y} _{k+1}=v_{k+1},\\
{\rm else}  \  \hat{x} _{k+1}=x_{k+1}, \hat{y} _{k+1}=y_{k+1}.
\end{cases}
\end{equation}
Compared with the traditional extrapolation algorithm, the main difference is to ensure that the algorithm is a monotone method in terms of objective function value, while general extrapolation algorithms may be nonmonotonic.
\par Bregman distance is a useful substitute for a distance, obtained from the various choices of functions. The applications of Bregman distance instead of the norm gives us alternative ways for more flexibility in the selection of regularization. Choosing appropriate Bregeman distances can obtain colsed form of solution for solving some subproblem. Bregman distance regularization is also an effective way to improve the numerical results of the algorithm. In {\cite{ZQ}}, the authors constructed the following two-step inertial Bregman alternating minimization algorithm using the information of the previous three iterates: 
\begin{equation}
\label{TiBAM}
\begin{cases}
x_{k+1}\in \arg\min_{ x\in \mathbb{R}^n}\{L(x,y_k)+D_{\phi_1}(x,x_k)+\alpha_{1k} \langle x,x_{k-1}-x_k\rangle+\alpha_{2k} \langle x,x_{k-2}-x_{k-1}\rangle\},\\
y_{k+1}\in \arg\min_{y\in \mathbb{R}^m}\{L(x_{k+1},y)+D_{\phi_2}(y,y_k)+\beta_{1k} \langle y,y_{k-1}-y_k\rangle+\beta_{2k} \langle y,y_{k-2}-y_{k-1}\},
\end{cases}
\end{equation}
where $D_{\phi_i}(i=1,2)$ denotes the Bregman distance with respect to  $\phi_i(i=1,2)$, respectively. The convergence is obtained provided an appropriate regularization of the objective function satisfies the Kurdyka--{\L}ojasiewicz inequality. Based on alternating minimization algorithm, {Zhao} et al. \cite{ZA} proposed  the following inertial alternating minimization with Bregman distance (BIAM) algorithm:
\begin{equation}
\label{BIAM}
\begin{cases}
x_{k+1}\in \arg\min_{ x\in \mathbb{R}^n}\{f(x)+Q(x,\hat{y} _k)+\lambda _{k} D_{\phi_1}(x,\hat{x} _k)\},\\
\hat{x}_{k+1}=x_{k+1}+\alpha(x_{k+1}-\hat x_{k}), \alpha \in [0,1),\\
y_{k+1}\in \arg\min_{ y\in \mathbb{R}^m}\{g(y)+Q(\hat{x}_{k+1},y)+\mu _{k} D_{\phi_2}(y,\hat{y} _k)\},\\
\hat{y}_{k+1}=y_{k+1}+\beta(y_{k+1}-\hat y_{k}), \beta \in [0,1).
\end{cases}
\end{equation}
Suppose that the benefit function satisfies the Kurdyka--{\L}ojasiewicz property and the parameters are selected appropriately, they proved the convergence of BIAM algorithm.
\par In this paper, based on the alternating structure-adapted proximal gradient methods, we combine inertial extrapolation technique and Bregman distance to construct two-step inertial Bregman alternating structure-adapted proximal gradient descent algorithm. And in order to make the proposed algorithm more effective and flexible, we also use some strategies to update the extrapolation parameter and solve the problems with unknown Lipschitz constant. Under some assumptions about the penalty parameter and objective function, the convergence of the proposed algorithm is obtained based on the Kurdyka--{\L}ojasiewicz property.  Moreover, we report some preliminary numerical results on involving quadratic programming and logistic regression problem to show the  feasibility and effectiveness of the proposed method. 
\par The article is organized as follows. In Section \ref{sect2}, we recall some concepts and important lemmas which will be used in the proof of main results. In Section \ref{sect3}, we present the Two-step inertial Bregman alternating structure-adapted proximal gradient algorithm and show its convergence. Finally, in Section \ref{sect4}, the preliminary numerical examples on nonconvex quadratic programming and sparse logistic regression problem are provided to illustrate the behavior of the proposed algorithm.
\section{Preliminaries}\label{sect2}
\hspace*{\parindent} Consider the Euclidean vector space $\mathbb{R}^d$ of dimension $d\geq 1$, the standard inner product and the induced norm on $\mathbb{R}^d$ are denoted by $\langle\cdot,\cdot\rangle$ and $\|\cdot\|_2$, respectively. We use $\omega(x_k)=\{x:\exists x_{k_j} \rightarrow x\}$ to stand for the limit set of $\{x_k\}_{k\in \mathbb{N}}$.
\par 
The {\it domain} of $f$ are defined by dom$f:=\{x\in \mathbb{R}^d: f(x)<+\infty\}$. We say that $f$ is {\it proper} if dom$f\neq \emptyset$, and $f$ is called {\it lower semicontinuous} at $x$ if $f(x)\le \lim\inf_{k\to \infty }f(x_k)$ for every sequence $\left \{ x_k \right \} $ converging to $x$. If $f$ is lower semicontinuous in its domain, we say $f$ is a lower semicontinuous function. If dom$f$ is closed and $f$ is lower semicontinuous over dom$f$ , then $f$ is a closed function. Further we recall some generalized subdifferential notions and the basic properties which are needed in this paper.
\subsection{Subdifferentials}
\begin{definition}
\rm
\label{Def211}(Subdifferentials)
Let $f: \mathbb{R}^d \rightarrow(-\infty,+\infty]$ be a proper and lower semicontinuous function.

(i)For $x\in$ {\rm dom}$f $, the Fr\'{e}chet subdifferential of $f$ at $x$, written  $\hat{\partial}f(x)$, is the set of vectors $v\in \mathbb{R}^d$ which satisfy
$$\liminf_{y\rightarrow x}\frac{1}{\|x-y\|_2}[f(y)-f(x)-\langle v,y-x\rangle]\geq 0.$$

(ii)If $x\not\in$ {\rm dom}$f$, then $\hat{\partial}f(x)=\emptyset$.  The limiting-subdifferential \cite{MV}, or simply the subdifferential for short, of $f$ at $x\in$ {\rm dom}$f$, written $\partial f(x)$, is defined as follows:
$$\partial f(x):=\{v\in \mathbb{R}^d: \exists x_k\rightarrow x, f(x_k)\rightarrow f(x), v_k\in \hat{\partial}f(x_k), v_k\rightarrow v\}.$$
\end{definition}

\begin{remark}
\label{rem21}
\rm
(a) The above definition implies that $\hat{\partial}f(x)\subseteq \partial f(x)$ for each $x\in \mathbb{R}^d$, where the first set is convex and closed while the second one is closed (see\cite{RWA}).

(b) (Closedness of $\partial f$)  Let $\{x_k\}_{k\in \mathbb{N}}$ and $\{v_k\}_{k\in \mathbb{N}}$ be sequences in $\mathbb{R}^d$ such that $v_k \in \partial f(x_k)$ for all $k\in \mathbb{N}$. If $(x_k,v_k)\rightarrow (x,v)$ and $f(x_k)\rightarrow f(x)$ as $k\rightarrow \infty$, then $v \in \partial f (x)$.


(c) If $f: \mathbb{R}^d \rightarrow(-\infty,+\infty]$ be a proper and lower semicontinuous and $h : \mathbb{R}^d \rightarrow \mathbb{R}$ is a continuously differentiable function, then $\partial(f+h)(x) = \partial f (x)+\nabla h(x)$ for all $x \in \mathbb{R}^d$.
\end{remark}
In what follows, we will consider the problem of finding a critical point $(x^\ast,y^\ast)\in$dom$L$.
\begin{lemma}{\rm (Fermat’s rule\cite{RW})}
\label{lem22}
Let $f: \mathbb{R}^{d}\rightarrow \mathbb{R}\cup \left \{ +  \infty  \right \}$ be a proper lower semicontinuous function. If $f$ has a local minimum at $x^{\ast }$, then $0\in \partial f(x^{\ast })$.
\end{lemma}
\par We call $x^{\ast }$ is a critical point of $f$ if \ $0\in \partial f(x^{\ast })$. The set of all critical points of $f$ denoted by crit$f$. 
\begin{lemma}{\rm (Descent lemma\cite{BT})}
\label{lem23}
Let $h: \mathbb{R}^{d}\rightarrow \mathbb{R}$ be a continuously differentiable function with gradient $\nabla h$ assumed $L_{h}$-Lipschitz continuous. Then
\begin{equation}
\label{(2.2)}
h\left ( u \right ) \le h\left ( v \right ) +   \left \langle u-v,\nabla h\left ( v \right )  \right \rangle+\frac{L_{h} }2{\left \| u-v \right \| ^{2}_2 ,\ \forall u,v\in \mathbb R^{d}. }
\end{equation}
\end{lemma} 
\subsection{The Kurdyka--{\L}ojasiewicz property}
\ \par In this section, we recall the K{\L} property, which plays a central role in the convergence analysis.
%
%
%
\begin{definition}
\rm
\label{Def21}(Kurdyka--{\L}ojasiewicz property \cite{ABR})
Let $f: \mathbb{R}^d \rightarrow(-\infty,+\infty]$ be a proper and lower semicontinuous function.

(i)The function  $f: \mathbb{R}^d \rightarrow(-\infty,+\infty]$  is said to have the Kurdyka--{\L}ojasiewicz (K{\L}) property at $x^\ast\in$dom$f$ if there exist $\eta\in (0,+\infty]$, a neighborhood $U$ of $x^\ast$ and a continuous concave function $\varphi:[0,\eta)\rightarrow \mathbb{R}_{+}$ such that $\varphi(0)=0$, $\varphi$ is $C^1$ on $(0,\eta)$, for all $s\in(0,\eta)$ it is $\varphi'(s)>0$ and for all $x$ in $U\cap[f(x^\ast)<f<f(x^\ast)+\eta]$ the Kurdyka--{\L}ojasiewicz  inequality holds, $$\varphi'(f(x)-f(x^\ast)){\rm dist}(0,\partial f(x))\geq 1.$$

(ii)Proper lower semicontinuous functions which satisfy the Kurdyka--{\L}ojasiewicz inequality at each point of {its domain} are called K{\L} functions.
\end{definition}

\begin{lemma}{\rm (Uniformized K{\L} property\cite{RW})}
\label{lem21}
Let $\Omega $ be a compact set and let $f: \mathbb{R}^{d}\rightarrow \mathbb{R}\cup \left \{ +  \infty  \right \}$ be a proper and lower semicontinuous function. Assume that $f$ is constant on $\Omega $ and satisfies the K{\L} property at each point of $\Omega $. Then, there exist $\varepsilon > 0,\eta> 0$ and $\varphi \in \Phi _{\eta } $ such that for all $x^\ast \in \Omega $ and for all $x\in \left \{ x\in\mathbb{R}^d :{\rm dist}(x,\Omega )< \varepsilon \right \} \cap[f(x^\ast)<f<f(x^\ast)+\eta]$, one has
$$\varphi'(f(x)-f(x^\ast)){\rm dist}(0,\partial f(x))\geq 1.$$ 
\end{lemma}
There is a broad class of functions satisfy the K{\L} property, such as strongly convex functions, real analytic functions, semi-algebraic functions \cite{ABR}, subanalytic functions \cite{WCX}, log-exp functions and so on.
\subsection{Bregman distance}
\begin{definition}
\rm
{ A function $f$ is said convex if dom$f$ is a convex set and if, for all $x$, $y\in$dom$f$, $\alpha\in[0,1]$,
$$f(\alpha x+(1-\alpha)y)\leq \alpha f(x)+(1-\alpha)f(y).$$
$f$ is said $\theta$-strongly convex with $\theta> 0$ if $f-\frac{\theta}{2}\|\cdot\|^2$ is convex, i.e.,
$$f(\alpha x+(1-\alpha)y)\leq \alpha f(x)+(1-\alpha)f(y)-\frac{1}{2}\theta\alpha(1-\alpha)\|x-y\|^2$$
for all $x$, $y\in$dom$f$ and  $\alpha\in[0,1]$.}
\end{definition}
Suppose that the function $f$ is differentiable. Then $f$ is convex if and only if dom$f$ is a convex set and
$$f(x)\ge f(y)+\langle \nabla f(y),x-y\rangle$$
holds for all $x$, $y\in$dom$f$. Moreover, $f$ is $\theta$-strongly convex with $\theta> 0$ if and only if
$$f(x)\ge f(y)+\langle \nabla f(y),x-y\rangle+\frac{\theta}{2}\|x-y\|^2$$
for all $x$, $y\in$dom$f$.
\begin{definition}
\rm
\label{Defbreg}
Let $\phi:\mathbb{R}^d \rightarrow(-\infty,+\infty]$ be a convex and G\^{a}teaux differentiable function.
The function $D_\phi :$ dom$\phi\,\,\times$ intdom$\phi \rightarrow [0,+\infty)$, defined by
$$D_\phi(x,y)=\phi(x)-\phi(y)-\langle \nabla\phi(y),x-y\rangle,$$
is called the Bregman distance with respect to $\phi$.
\end{definition}
{From the above definition,  it follows  that
\begin{equation}
\label{(2.1)}
D_\phi(x,y)\geq\frac{\theta}{2}\|x-y\|^2,
\end{equation}
if $\phi$ is $\theta$-strongly convex. 
\section{Algorithm and convergence analysis}\label{sect3}
\begin{Assumption}
\label{Assumption31}
\rm
(i) $L:\mathbb R^{n}\times \mathbb R^{m}\to \mathbb R\cup \left \{ \infty  \right \} $ is lower bounded.

(ii) $f:\mathbb R^{n}\to \mathbb R$ and $g:\mathbb R^{m}\to \mathbb R$ are continuously differentiable and their gradients $\nabla f$ and $\nabla g$ are Lipschitz continuous with constants $L_{\nabla f}$, $L_{\nabla g}$ respectively.

(iii) $Q:\mathbb R^{n}\times \mathbb R^{m}\to \mathbb R\cup \left \{ \infty  \right \} $ is a proper, lower semicontinuous.

(iv) $\phi_i(i=1,2)$ is $\theta_i$-strongly convex differentiable function, $\theta_1>L_{\nabla f}$, $\theta_2>L_{\nabla g}$.
And the gradient $\nabla \phi_i$ is $\eta_i$-Lipschitz continuous, i.e.,
\begin{equation}
\label{(33.2)}
\aligned
&\left \| \nabla{\phi_1}(x) -\nabla{\phi_1}(\hat{x})\right \|\le \eta_1 \|x-\hat{x}\|,\ \forall x,\hat{x}\in \mathbb R^{n},\\
&\left \| \nabla{\phi_2}(y) -\nabla{\phi_2}(\hat{y})\right \|\le \eta_2 \|y-\hat{y}\|,\ \forall y,\hat{y}\in \mathbb R^{m}. 
\endaligned
\end{equation}
\end{Assumption}
\subsection{The proposed algorithm}

\begin{algorithm}[H]
\caption{TiBASAP: Two-step inertial Bregman alternating structure-adapted proximal gradient descent algorithm}\label{alg1}
\begin{algorithmic}
\Require 
Take $(x_0,y_0)\in \mathbb{R}^n\times \mathbb{R}^m, (\hat{x} _0,\hat{y} _0)=(x_0,y_0), \alpha  _k\in [0,\alpha _{\max}], \beta _k\in [0,\beta _{\max}], \alpha  _{\max}+\beta _{\max}<1, k=0.$\\
1. \noindent{\bf Compute}
\begin{equation}
\label{(3.1)}
\begin{cases}
\aligned
x_{k+1}\in \arg\min_{ x\in \mathbb{R}^n}\{&Q(x,\hat{y} _k)+\langle \nabla f(\hat{x}_k),x\rangle+D_{\phi_1}(x,\hat{x} _k)\},\\
y_{k+1}\in \arg\min_{y\in \mathbb{R}^m}\{&Q(x_{k+1},y)+\langle \nabla g(\hat{y} _k),y\rangle+D_{\phi_2}(y,\hat{y} _k)\}.
\endaligned
\end{cases}
\end{equation}
2.
\begin{equation}
\label{(3.2)}
\binom{u_{k+1}}{v_{k+1}} =\binom{x_{k+1}}{y_{k+1}}+\alpha _k\binom{x_{k+1}-x_k}{y_{k+1}-y_k}+\beta_k\binom{x_{k}-x_{k-1}}{y_{k}-y_{k-1}}.
\end{equation}
3. \noindent{\bf If} $L(u_{k+1},v_{k+1})\le L(x_{k+1},y_{k+1})$, \noindent{\bf then}
\begin{equation}
\label{(3.3)}
\hat{x} _{k+1}=u_{k+1},\ \hat{y} _{k+1}=v_{k+1},
\end{equation}
\noindent{\bf else}
\begin{equation}
\label{(3.4)}
\hat{x} _{k+1}=x_{k+1},\ \hat{y} _{k+1}=y_{k+1}.
\end{equation}
4. \noindent{\bf Set} $k\gets k+1$, go to step1.
\end{algorithmic}
\end{algorithm}
\bigskip


{\begin{remark}\rm We discuss the relation of Algorithm \ref{alg1} to the other existing algorithms from literatures.
\begin{itemize}
\item[(i)]
If we take $\phi_1(x)=\frac{1}{2\sigma}\|x\|^2_2$ and $\phi_2(y)=\frac{1}{2\tau}\|y\|^2_2$ for all $x\in \mathbb{R}^n$ and $y\in \mathbb{R}^m$, then Algorithm \ref{alg1} becomes the following iterative method:
\begin{equation}
\label{(3.0)}
\begin{cases}
x_{k+1}\in \arg\min_{ x\in \mathbb{R}^n}\{Q(x,\hat{y} _k)+\langle \nabla f(\hat{x}_k),x\rangle+\frac{1}{2\sigma}\|x-\hat{x} _k\|^2_2\},\\
y_{k+1}\in \arg\min_{y\in \mathbb{R}^m}\{Q(x_{k+1},y)+\langle \nabla g(\hat{y} _k),y\rangle+\frac{1}{2\tau}\|y-\hat{y}_k\|^2_2\},\\
u_{k+1}=x_{k+1}+\alpha_{k}(x_{k+1}-x_{k})+\beta_{k}(x_{k}-x_{k-1}),\\
v_{k+1}=y_{k+1}+\alpha_{k}(y_{k+1}-y_{k})+\beta_{k}(y_{k}-y_{k-1}),\\
{\rm if}  \ L(u_{k+1},v_{k+1})\le L(x_{k+1},y_{k+1}), \ {\rm then} \ \hat{x} _{k+1}=u_{k+1}, \hat{y} _{k+1}=v_{k+1},\\
{\rm else}  \  \hat{x} _{k+1}=x_{k+1}, \hat{y} _{k+1}=y_{k+1}.
\end{cases}
\end{equation}
\item[(ii)]
Letting $\beta_{k}\equiv 0$ for all $k\geq 0$, \eqref{(3.0)} becomes the  accelerated alternating structure-adapted proximal gradient descent (aASAP) algorithm \eqref{aASAP}.
\item[(iii)]
Letting $\alpha_{k}\equiv\beta_{k}\equiv 0$ for all $k\geq 0$,  \eqref{(3.0)} becomes the alternating structure-adapted proximal gradient descent (ASAP) algorithm \eqref{ASAP}.

\end{itemize}
\end{remark}}
{\begin{remark}\rm
Compared with the traditional extrapolation algorithm, the main difference is step 3 which ensures the algorithm is a monotone method in terms of objective function value, while general extrapolation algorithms may be nonmonotonic.
\end{remark}}
\par For extrapolation parameters $\alpha  _k$ and $\beta _k$, there are at least two ways to choose them, either as costant or by dynamic update.
For example, in \cite{LLA,ZZ} it was defined as
\begin{equation}
\label{fista}
\begin{cases}
\alpha  _k=\beta_k=\frac{t_{k-1}-1}{2t_k},\\
t_{k+1}=\frac{1+\sqrt{1+4t_{k}^{2}}}{2},\\
\end{cases}
\end{equation}
where $t_{-1}=t_0=1$. In order to make Algorithm \ref{alg1} more effective, we present an adaptive method to update the extrapolation parameter $\alpha  _k$, $\beta _k$, which is given in Algorithm \ref{alg2}.
\begin{algorithm}[H]
\caption{Two-step inertial Bregman alternating structure-adapted proximal gradient descent with adaptive extrapolation parameter algorithm}\label{alg2}
\begin{algorithmic}
\Require 
Take $(x_0,y_0)\in \mathbb{R}^n\times \mathbb{R}^m,\ (\hat{x} _0, \hat{y} _0)=(x_0,y_0), \ \alpha  _0\in [0,\alpha _{\max}], \ \beta _0\in [0,\beta _{\max}], \ \alpha  _{\max}+\beta _{\max}<1, \ t>1, \ k=0.$\\
1. \noindent{\bf Compute}
\begin{equation}
\label{(3.5)}
\begin{cases}
\aligned
x_{k+1}\in \arg\min_{ x\in \mathbb{R}^n}\{&Q(x,\hat{y} _k)+\langle \nabla f(\hat{x}_k),x\rangle+D_{\phi_1}(x,\hat{x} _k)\},\\
y_{k+1}\in \arg\min_{y\in \mathbb{R}^m}\{&Q(x_{k+1},y)+\langle \nabla g(\hat{y} _k),y\rangle+D_{\phi_2}(y,\hat{y} _k)\}.
\endaligned
\end{cases}
\end{equation}
2.
\begin{equation}
\label{(3.6)}
\binom{u_{k+1}}{v_{k+1}} =\binom{x_{k+1}}{y_{k+1}}+\alpha _k\binom{x_{k+1}-x_k}{y_{k+1}-y_k}+\beta_k\binom{x_{k}-x_{k-1}}{y_{k}-y_{k-1}}.
\end{equation}
3. \noindent{\bf If} $L(u_{k+1},v_{k+1})\le L(x_{k+1},y_{k+1})$, \noindent{\bf then}
\begin{equation}
\label{(3.7)}
\aligned
&\hat{x} _{k+1}=u_{k+1},\ \hat{y} _{k+1}=v_{k+1},\\
&\alpha_{k+1}=\min\left \{ t\alpha _k,\alpha _{\max} \right \}, \\
&\beta_{k+1}=\min\left \{ t\beta _k,\beta _{\max} \right \}, 
\endaligned
\end{equation}
\noindent{\bf else}
\begin{equation}
\label{(3.8)}
\aligned
&\hat{x} _{k+1}=x_{k+1},\ \hat{y} _{k+1}=y_{k+1},\\
&\alpha _{k+1}=\frac{\alpha _k}{t},\\
&\beta _{k+1}=\frac{\beta _k}{t}.  
\endaligned
\end{equation}
4. \noindent{\bf Set} $k\gets k+1$, go to step1.
\end{algorithmic}
\end{algorithm}
\bigskip
{\begin{remark}\rm
Compared with constant or dynamic update by \eqref{fista}, the adaptive extrapolation parameters $\alpha  _k$ and $\beta _k$ in Algorithm \ref{alg2} may adopt more extrapolation steps. The numerical results will verify the effectiveness of the adaptive strategy in Section \ref{sect4}.
\end{remark}}
{\begin{remark}\rm
\par A possible drawback of Algorithm \ref{alg1} and Algorithm \ref{alg2} is that the Lipschitz constants $L_{\nabla f}$, $L_{\nabla g}$ are not always known or computable. However, the Lipschitz constant determines the strong convexity modulus range of $\phi_1$ and $\phi_2$. Especially, if $D_{\phi_i}=\frac{1}{2} \left \| \cdot  \right \| ^2,\ (i=1,2)$, the Lipschitz constant $L_{\nabla f}$, $L_{\nabla g}$ determines the range of stepsize $\sigma, \tau$ in \eqref{(3.0)}. Even if the Lipschitz constant is known, it is always large in general, which makes the strong convexity modulus $\theta_1$, $\theta_2$ very small. Therefore, in order to improve the efficiency of the algorithm, a method for estimating a proper local Lipschitz constant will be given below.
\end{remark}}
A backtracking method is proposed to evaluate the local Lipschitz constant. We adopt Barzilai-Borwein (BB) method \cite{BB} with lower bound to initialize the stepsize at each iteration. The procedure of computing the stepsize of the $x$-subproblem is shown in the following box.\\
\noindent \fbox{\parbox{\textwidth}{%
\noindent{\bf Backtracking with BB method with lower bound}
\par \ \ \noindent{\bf Compute the BB stepsize:} set $t_{\min}>0$,
\begin{equation*}
\aligned
&l_k=\nabla f(x_k)-\nabla f(\hat{x}_{k-1}),\\
&s_k=x_k-\hat{x}_{k-1},\\
&t_k=\max\left \{ \frac{\left | s_{k}^{T} l_k \right | }{ s_{k}^{T} s_k},t_{\min}  \right \}.
\endaligned
\end{equation*}
\noindent{\bf Backtracking}: set $\rho>1, \delta >0$.
\par \ \ \noindent{\bf Repeat compute}:
\begin{equation*}
\aligned
&x_{k+1}\in \arg\min_{ x\in \mathbb{R}^n}\{Q(x,\hat{y} _k)+\langle \nabla f(\hat{x}_k),x\rangle+t_kD_{\phi_1}(x,\hat{x} _k)\},\\
&t_k\gets \rho t_k.  
\endaligned
\end{equation*}
\par \ \ \noindent{\bf until}
\begin{equation}
\label{(3.12)}
Q(x_{k+1},\hat{y} _k)+f(x_{k+1})\le Q(\hat{x}_k,\hat{y} _k)+f(\hat{y} _k)-\frac{\delta}{2}\|x_{k+1}-\hat{x} _k\|^2_2.
\end{equation}
}}
{\begin{remark}\rm
The stepsize of $y$-subproblem can be obtained similarly. The process of enlarging $t_k$ is actually approaching the local Lipschitz constant of $f$. Since $f$ is gradient Lipschitz continuous, the backtracking process can be terminated in finite steps for achieving a suitable $t_k$.
\end{remark}}
\subsection{Convergence analysis}
\ \par In this section, we will prove the convergence of Algorithm \ref{alg1}. Note that the bound of $\alpha _k$ and $\beta_k$ is no more than $\alpha _{\max}$ and $\beta _{\max}$ in Algorithm \ref{alg2}, respectively. So the convergence properties of Algorithm \ref{alg1} are also applicable for Algorithm \ref{alg2}. 

\par Under Assumption \ref{Assumption31}, some convergence results will be proved (see Lemma \ref{lem31}). We will also consider the following additional assumptions to establish stronger convergence results.
\begin{Assumption}
\label{Assumption32}
\rm
(i) $L$ is coercive and the domain of $Q$ is closed.

(ii) The subdifferential of $Q$ obeys:
$$ \forall (x,y)\in {\rm dom} \ Q, \partial _xQ(x,y)\times \partial _yQ(x,y)\subset \partial Q(x,y).$$

(iii) $Q:\mathbb R^{n}\times \mathbb R^{m}\to \mathbb R\cup \left \{ \infty  \right \} $ has the following form
$$Q(x,y)=q(x,y)+h(x),$$
where $h:\mathbb R^{n}\to \mathbb R\cup \left \{ \infty  \right \} $ is continuous on its domain; $q:\mathbb R^{n}\times \mathbb R^{m}\to \mathbb R\cup \left \{ \infty  \right \} $ is a continuous function on dom $Q$ such that for any $y$, the partial function $q(\cdot ,y )$ is continuously differentiable about $x$. Besides, for each bounded subset $D_1\times D_2\subset {\rm dom}Q$, there exists $\xi >0$, such that for any $\bar x\in D_1$, $(y,\bar y)\in D_2\times D_2$, it holds that
$$\left \| \nabla _xq(\bar x,y)- \nabla _xq(\bar x,\bar y)\right \| \le \xi \left \| y-\bar y \right \| .$$
\end{Assumption}
\begin{remark}
\label{rem31}
\rm
(i) Assumption \ref{Assumption32}(i) ensure that the sequences generated by our proposed algorithms is bounded which plays an important role in the proof of convergence.
\par(ii) Assumption \ref{Assumption32}(ii) is a generic assumption for the convergence of alternating schemes. Because $f$ and $g$ are continuously differentiable, we have $\partial _xL(x,y)\times \partial _yL(x,y)\subset \partial L(x,y)$ by Remark \ref{rem21} (c).
\par(iii) From Assumptions \ref{Assumption31} and \ref{Assumption32}, $L(x,y)$ is continuous on its domain, equal to dom $Q$, which is nonempty and closed. For any sequence $\left \{ (x_k, y_k ) \right \} $ converges to $(\bar x, \bar y)$, it holds that $\left \{L (x_k, y_k ) \right \} $ converges to $L(\bar x, \bar y)$.
\end{remark}
\begin{lemma}
\label{lem31}
{\it Suppose that Assumption \ref{Assumption31} hold. Let $\left \{ (x_k, y_k ) \right \} $ and $\left \{ (\hat{x}_k, \hat{y}_k ) \right \} $ are sequences generated by Algorithm \ref{alg1}. The following assertions hold.

{\rm (i)} The sequences $\left \{ L(x_k, y_k ) \right \}$, $\left \{ L(\hat{x}_k, \hat{y}_k ) \right \}$ are monotonically nonincreasing and have the same limiting point, i.e.
$$\lim_{k \to \infty } L(x_k, y_k )=\lim_{k \to \infty } L(\hat{x}_k, \hat{y}_k )=L^\ast. $$
In particular,
\begin{equation}
\label{(3.9)}
\aligned
&L(x_{k+1}, y_{k+1} )\le L(x_k, y_k )-\rho \left \| (x_{k+1}-\hat{x}_k,y_{k+1}-\hat{y}_k) \right \|^2 ,\\
&L(\hat{x}_{k+1}, \hat{y}_{k+1} )\le L(\hat{x}_k, \hat{y}_k )-\rho \left \| (x_{k+1}-\hat{x}_k,y_{k+1}-\hat{y}_k) \right \|^2,
\endaligned
\end{equation}
where $\rho=\min\left \{ \frac{\theta _{1}-L_{\nabla f}}{2}, \frac{\theta _{2}-L_{\nabla g}}{2} \right \} .$

{\rm (ii)} $\sum_{k=0}^{\infty } \left \| (x_{k+1}-\hat{x}_k,y_{k+1}-\hat{y}_k) \right \|^2< \infty ,$ and hence
$$\lim_{k \to \infty }\left \| x_{k+1}-\hat{x}_k \right \| =0, \ \lim_{k \to \infty }\left \| y_{k+1}-\hat{y}_k \right \| =0.$$
}
\end{lemma}
\begin{proof} (i) 
From Lemma \ref{lem23}, we have
\begin{equation}
\label{(3.10)}
f(x_{k+1} )\le f(\hat {x}_{k} )+\langle \nabla f(\hat{x}_k),x _{k+1}-\hat{x}_k\rangle+\frac{L_{\nabla f}}{2}\left \|x _{k+1}-\hat{x}_k  \right \|^2.
\end{equation}
According to $x$-subproblem in the iterative scheme \eqref{(3.1)}, we obtain
$$Q(x_{k+1},\hat{y} _k)+\langle \nabla f(\hat{x}_k),x_{k+1}\rangle+D_{\phi_1}(x_{k+1},\hat{x} _k)\le Q(\hat{x}_k,\hat{y} _k)+\langle \nabla f(\hat{x}_k),\hat{x}_k\rangle+D_{\phi_1}(\hat{x}_k,\hat{x} _k),$$
which implies that\\
\begin{equation}
\label{(3.11)}
Q(x_{k+1},\hat{y} _k)+\langle \nabla f(\hat{x}_k),x_{k+1}-\hat{x} _k\rangle+D_{\phi_1}(x_{k+1},\hat{x} _k)\le Q(\hat{x}_k,\hat{y} _k).
\end{equation}
By \eqref{(2.1)}, it follows from \eqref{(3.10)} and \eqref{(3.11)} that\\
\begin{equation}
\label{(3.12)}
\aligned
f(x_{k+1} )+Q(x_{k+1},\hat{y} _k)
&\le f(\hat {x}_{k} )+Q(\hat{x}_k,\hat{y} _k)+\frac{L_{\nabla f}}{2}\left \|x _{k+1}-\hat{x}_k  \right \|^2-D_{\phi_1}(x_{k+1},\hat{x} _k)\\
&\le f(\hat {x}_{k} )+Q(\hat{x}_k,\hat{y} _k)-\frac{\theta _1-L_{\nabla f}}{2}\left \|x _{k+1}-\hat{x}_k  \right \|^2.
\endaligned
\end{equation}
Similarly, for $y$-subproblem, we can get\\
\begin{equation}
\label{(3.13)}
g(y_{k+1} )+Q(x_{k+1},y _{k+1})
\le g(\hat {y}_{k} )+Q(x_{k+1},\hat{y} _k)-\frac{\theta _2-L_{\nabla g}}{2}\left \|y _{k+1}-\hat{y}_k  \right \|^2.
\end{equation}
Adding \eqref{(3.12)} and \eqref{(3.13)}, we have\\
\begin{equation}
\label{(3.14)}
\aligned
&f(x_{k+1} )+Q(x_{k+1},y _{k+1})+g(y_{k+1} )\\
\le &f(\hat {x}_{k} )+Q(\hat{x}_k,\hat{y} _k)+ g(\hat {y}_{k} )-\frac{\theta _1-L_{\nabla f}}{2}\left \|x _{k+1}-\hat{x}_k  \right \|^2-\frac{\theta _2-L_{\nabla g}}{2}\left \|y _{k+1}-\hat{y}_k  \right \|^2\\
\le &f(\hat {x}_{k} )+Q(\hat{x}_k,\hat{y} _k)+ g(\hat {y}_{k} )-\rho(\left \|x _{k+1}-\hat{x}_k  \right \|^2+\left \|y _{k+1}-\hat{y}_k  \right \|^2),
\endaligned
\end{equation}
which can be abbreviated as
\begin{equation}
\label{(3.15)}
L(x_{k+1},y _{k+1})\le L(\hat{x}_k,\hat{y} _k)-\rho\left \| (x_{k+1}-\hat{x}_k,y_{k+1}-\hat{y}_k) \right \|^2.
\end{equation}
By the choice of \eqref{(3.3)} and \eqref{(3.4)} for $(\hat{x}_k,\hat{y} _k)$  in Algorithm \ref{alg1}, we get\\
\begin{equation}
\label{(3.16)}
\aligned
L(\hat{x}_{k+1},\hat{y}_{k+1})\le L(x_{k+1},y _{k+1})
&\le L(\hat{x}_k,\hat{y} _k)-\rho \left \| (x_{k+1}-\hat{x}_k,y_{k+1}-\hat{y}_k) \right \|^2\\
&\le L(x_k,y _k)-\rho \left \| (x_{k+1}-\hat{x}_k,y_{k+1}-\hat{y}_k) \right \|^2.
\endaligned
\end{equation}
Hence $\left \{  L(x_k,y _k)\right \}$ and $\left \{  L(\hat{x}_k,\hat{y} _k)\right \}$ are nonincreasing sequences, and\\
\begin{equation}
\label{(3.17)}
L(x_{k+1},y _{k+1})\le L(\hat{x}_k,\hat{y} _k)\le L(x_k,y _k)
\end{equation}
holds. Furthermore, $L$ is lower bounded according to Assumption \ref{Assumption31}, therefore the sequence ${L(x_k,y _k)}$ is convergent.  Let $\lim_{k \to \infty} L(x_k,y _k)=L^\ast $. Taking limit on \eqref{(3.17)}, and using the squeeze theorem, we get $\lim_{k \to \infty} L(\hat{x}_k,\hat{y} _k)=L^\ast$. 

(ii) It follows from \eqref{(3.16)} that
\begin{equation}
\label{(3.18)}
\rho \left \| (x_{k+1}-\hat{x}_k,y_{k+1}-\hat{y}_k) \right \|^2\le L(x_k,y _k)-L(x_{k+1},y _{k+1}).
\end{equation}
Summing up \eqref{(3.18)} from $k= 0,\cdots,K$, we obtain\\
\begin{equation}
\label{(3.19)}\
\aligned
\sum_{k=0}^{K} \rho \left \| (x_{k+1}-\hat{x}_k,y_{k+1}-\hat{y}_k) \right \|^2
&\le L(x_0,y _0)-L(x_{K+1},y _{K+1})\\
&\le L(x_0,y _0)-L^\ast < + \infty.
\endaligned
\end{equation}
Let $K\to \infty $, it follows that $\lim_{k \to \infty} \left \|(x _{k+1}-\hat{x}_k,y_{k+1}-\hat{y}_k)  \right \|^2=0$, and therefore\\
\begin{equation}
\label{(3.20)}
\lim_{k \to \infty} \left \|x _{k+1}-\hat{x}_k\right \|=0,\   \  \  \ \lim_{k \to \infty} \left \|y_{k+1}-\hat{y}_k\right \|=0.
\end{equation}
\end{proof}
\begin{lemma}
\label{lem32}
{\it Suppose that Assumption \ref{Assumption31} and Assumption \ref{Assumption32} hold. Let $\left \{  (x_k, y_k )\right \}$ and $\left \{  (\hat{x}_k,\hat{y} _k)\right \}$ be the sequences generated by Algorithm \ref{alg1}. For any integer $k\ge1$, set
\begin{equation}
\label{(3.25)}
p_{x}^{k+1} =\nabla_xq(x_{k+1},y_{k+1})-\nabla_xq(x_{k+1},\hat{y}_k)+q_{x}^{k+1},
\end{equation}
\begin{equation}
\label{(3.22)}
p_{y}^{k+1} =\nabla g(y_{k+1})-\nabla g( \hat{y}_k)-\nabla{\phi _2}(y_{k+1})+\nabla{\phi _2}(\hat{y}_k),
\end{equation}
where $q_{x}^{k+1} =\nabla f(x_{k+1})-\nabla f( \hat{x}_k)-\nabla{\phi _1}(x_{k+1})+\nabla{\phi _1}(\hat{x}_k)$, then $(p_{x}^{k+1},p_{y}^{k+1})\in \partial L(x_{k+1},y_{k+1})$, and there exists $\varrho>0$, such that
\begin{equation}
\label{(3.26)}
\left \|(p_{x}^{k+1},p_{y}^{k+1})  \right \| \le \varrho\left \|(x _{k+1}-\hat{x}_k,y_{k+1}-\hat{y}_k)  \right \|.
\end{equation}
}
\end{lemma}
\begin{proof}
From the iterative scheme \eqref{(3.1)}, we know\\
$$x_{k+1}\in \arg\min_{ x\in \mathbb{R}^n}\{Q(x,\hat{y} _k)+\langle \nabla f(\hat{x}_k),x\rangle+D_{\phi_1}(x,\hat{x} _k)\}.$$
By Fermat’s rule, $x_{k+1}$ satisfies\\
$$0\in \partial_xQ(x_{k+1},\hat{y}_k)+\nabla f( \hat{x}_k)+\nabla{\phi _1}(x_{k+1})-\nabla{\phi _1}(\hat{x}_k),$$
which implies that\\
\begin{equation}
\label{(3.23)}\
\aligned
\nabla f(x_{k+1})-\nabla f( \hat{x}_k)-\nabla{\phi _1}(x_{k+1})+\nabla{\phi _1}(\hat{x}_k)
&\in \partial_xQ(x_{k+1},\hat{y}_k)+\nabla f( x_{k+1})\\
&= \partial_xL(x_{k+1},\hat{y}_k).
\endaligned
\end{equation}
Similarly, by $y$-subproblem in iterative scheme \eqref{(3.1)}, we have\\
$$0\in \partial_yQ(x_{k+1},y_{k+1})+\nabla g( \hat{y}_k)+\nabla{\phi _2}(y_{k+1})-\nabla{\phi _2}(\hat{y}_k),$$
which implies that\\
\begin{equation}
\label{(3.24)}\
\aligned
\nabla g(y_{k+1})-\nabla g( \hat{y}_k)-\nabla{\phi _2}(y_{k+1})+\nabla{\phi _2}(\hat{y}_k)
&\in \partial_yQ(x_{k+1},y_{k+1})+\nabla g( y_{k+1})\\
&= \partial_yL(x_{k+1},y_{k+1}).
\endaligned
\end{equation}
From Assumption \ref{Assumption32} (iii), we know that\\
$$\partial_xQ(x,y)=\nabla _xq(x,y)+\partial h(x).$$
According to \eqref{(3.23)}, we have
\begin{equation}
\label{(3.27)}\
\aligned
q_{x}^{k+1}\in \partial_xL(x_{k+1},\hat{y}_k)
&=\partial_xQ(x_{k+1},\hat{y}_k)+\nabla f( x_{k+1})\\
&=\nabla _xq(x_{k+1},\hat{y}_k)+\partial h(x_{k+1})+\nabla f( x_{k+1}).
\endaligned
\end{equation}
It follows from \eqref{(3.25)} that
\begin{equation}
\label{(3.28)}\
\aligned
p_{x}^{k+1} &=\nabla_xq(x_{k+1},y_{k+1})-\nabla_xq(x_{k+1},\hat{y}_k)+q_{x}^{k+1}\\
&\in \nabla_xq(x_{k+1},y_{k+1})+\partial h(x_{k+1})+\nabla f( x_{k+1})\\
&= \partial_xQ(x_{k+1},y_{k+1})+\nabla f( x_{k+1})\\
&= \partial_xL(x_{k+1},y_{k+1}).
\endaligned
\end{equation}
Hence,\\
$$(p_{x}^{k+1},p_{y}^{k+1})\in \partial_xL(x_{k+1},y_{k+1})\times \partial_yL(x_{k+1},y_{k+1})\subset \partial L(x_{k+1},y_{k+1}).$$
Now we begin to estimate the norms of $p_{x}^{k+1}$ and $\ p_{y}^{k+1}$. Under Assumption \ref{Assumption32} (i) that $L$ is coercive, we deduce that $\left \{  (x_{k+1}, y_{k+1} )\right \}$ is a bounded set. Then from Assumption \ref{Assumption32} (iii) and \eqref{(33.2)}, we have\\
\begin{equation*}
\aligned
\left \|  p_{x}^{k+1} \right \| &\le \left \| \nabla_xq(x_{k+1},y_{k+1})-\nabla_xq(x_{k+1},\hat{y}_k) \right \| +\left \|  q_{x}^{k+1} \right \| \\
&\le \xi \left \| y_{k+1}-\hat{y}_k \right \|+\left \|  \nabla f(x_{k+1})-\nabla f( \hat{x}_k)\right \| +\left \| \nabla{\phi _1}(\hat{x}_k) -\nabla{\phi _1}(x_{k+1})\right \| \\
&\le \xi \left \| y_{k+1}-\hat{y}_k \right \|+L_{\nabla f}\left \| x_{k+1}-\hat{x}_k\right \| +\eta _1\left \| x_{k+1}-\hat{x}_k\right \| \\
&= \xi \left \| y_{k+1}-\hat{y}_k \right \|+\left (L_{\nabla f}+\eta _1\right )\left \| x_{k+1}-\hat{x}_k\right \|,\\
\left \|  p_{y}^{k+1} \right \| 
&\le \left \|  \nabla g(y_{k+1})-\nabla g( \hat{y}_k)\right \| +\left \| \nabla{\phi _2}(\hat{y}_k) -\nabla{\phi _2}(y_{k+1})\right \| \\
&\le L_{\nabla g}\left \| y_{k+1}-\hat{y}_k\right \| +\eta _2\left \| y_{k+1}-\hat{y}_k\right \| \\
&= \left (L_{\nabla g}+\eta _2\right )\left \| y_{k+1}-\hat{y}_k\right \|,
\endaligned
\end{equation*}
and hence\\
\begin{equation*}
\aligned
\left \|  (p_{x}^{k+1},p_{y}^{k+1}) \right \| ^2&= \left \|  p_{x}^{k+1}\right \| ^2+\left \|p_{y}^{k+1}\right \| ^2\\
&\le 2\left (L_{\nabla f}+\eta _1\right )^2\left \| x_{k+1}-\hat{x}_k\right \|^2+2\xi^2 \left \| y_{k+1}-\hat{y}_k \right \|^2+\left (L_{\nabla g}+\eta _2\right )^2\left \| y_{k+1}-\hat{y}_k\right \|^2\\
&= 2\left (L_{\nabla f}+\eta _1\right )^2\left \| x_{k+1}-\hat{x}_k\right \|^2+2\left [ \xi^2 +\left (L_{\nabla g}+\eta _2\right )^2\right ] \left \| y_{k+1}-\hat{y}_k \right \|^2.
\endaligned
\end{equation*}
The above inequality holds from the fact that $(a+b)^2\le 2(a^2+b^2),\forall a,b\in \mathbb{R}.$ 
\par Set $\varrho^{2} =\max\left \{ 2\left (L_{\nabla _f}+\eta _1\right )^2 ,2\xi^2 +2\left (L_{\nabla _g}+\eta _2\right )^2\right \} $. Then\\
\begin{equation*}
\aligned
&\left \|  (p_{x}^{k+1},p_{y}^{k+1}) \right \| ^2\le \varrho^{2}\left ( \left \| x_{k+1}-\hat{x}_k\right \|^2+\left \| y_{k+1}-\hat{y}_k \right \|^2 \right )  \\
&\left \|  (p_{x}^{k+1},p_{y}^{k+1}) \right \|\le \varrho\left \| \left (x_{k+1}-\hat{x}_k,y_{k+1}-\hat{y}_k\right ) \right \|.
\endaligned
\end{equation*}
So the conclusion holds.
\end{proof}
Below, we would summarize some properties about cluster points and prove every cluster point of a sequence generated by Algorithm \ref{alg1} is the critical point of $L$. For simplicity, we introduce the following notations. Let\\
$$z_k=(x_k,y_k),\ \omega _k=(\hat{x}_k,\hat{y}_k).$$
So $L(z_k)=L(x_k,y_k), L(\omega _k)=L(\hat{x}_k,\hat{y}_k).$
\par Let $\{z_k\}$ be the sequence generated by Algorithm \ref{alg1} with initial point $z_0$. Under Assumption \ref{Assumption32} (i) that $L$ is coercive, we deduce that $\{z_k\}$ is a bounded set, and it has at least one cluster point. The set of all cluster points is denoted by $\mathcal{L}(z_0)$, i.e.,
$$\mathcal{L}(z_0):=\{\hat{z}=(\hat{x},\hat{y})\in \mathbb{R}^n\times \mathbb{R}^m: \exists {\rm \ strictly\ increasing} \left \{ k_j \right \} _{j\in \mathbb{N}} 
{\rm\ such\ that\ } z_{k_j}\to\hat{ z}, j\to \infty \}.$$
\begin{lemma}
\label{lem34}
{\it Suppose Assumption \ref{Assumption31} and Assumption \ref{Assumption32} hold, let $\left \{  z _k\right \}$ be a sequence generated by Algorithm \ref{alg1} with initial point $z_0$. Then the following results hold.

{\rm (i)} $\mathcal{L}(z_0)$ is a nonempty compact set, and $L$ is finite and constant on $\mathcal{L}(z_0)$,

{\rm (ii)} $\mathcal{L}(z_0)\subset {\rm crit} \ L$,

{\rm (iii)} $\lim_{k \to \infty} {\rm dist}\left ( z_k,\mathcal{L}(z_0) \right ) =0$.
}
\end{lemma}
\begin{proof}
{\rm (i)} The fact $\left \{  z _k\right \}$  is bounded yields $\mathcal{L}(z_0)$ is nonempty. In addition, $\mathcal{L}(z_0)$ can be reformulated as an intersection of compact sets
$$\mathcal{L}(z_0)=\bigcap_{s\in \mathbb{N}}\overline{\bigcup_{k\ge s}{z_k}} ,$$
which illustrates that $\mathcal{L}(z_0)$ is a compact set.

For any $\hat{z}=(\hat{x},\hat{y})\in \mathcal{L}(z_0)$, there exists a subsequence $\left \{  z _{k_j}\right \}$ such that
$$\lim_{j \to \infty} z_{k_j}=\hat{z} .$$
Since $L$ is continuous, we have
$$\lim_{j \to \infty} L(z_{k_j})=L(\hat{z} ).$$
According to Lemma \ref{lem31}, we know that $\left \{  L(z _k)\right \}$ converges to $L^\ast $ globally. Hence
\begin{equation}
\label{(3.29)}
\lim_{j \to \infty} L(z_{k_j})=\lim_{k \to \infty} L(z_{k})=L(\hat{z} )=L^\ast.
\end{equation}
which means $L$ is a constant on $\mathcal{L}(z_0)$.

{\rm (ii)}  Take $\hat{z}\in \mathcal{L}(z_0)$, then $\exists \ \left \{ (x_{k_j},y_{k_j}) \right \}$ such that $(x_{k_j},y_{k_j})\to \hat{z}$. From \eqref{(3.20)}, we have
\begin{equation*}
\lim_{j \to \infty} \left \|x _{k_j}-\hat{x}_{k_{j-1}}\right \|=0,\   \  \  \ \lim_{j \to \infty} \left \|y_{k_j}-\hat{y}_{k_{j-1}}\right \|=0,
\end{equation*}
hence,
\begin{equation*}
\lim_{j \to \infty} x _{k_j}=\lim_{j \to \infty} \hat{x}_{k_{j-1}}=\hat{x},\   \  \  \ \lim_{j \to \infty} y _{k_j}=\lim_{j \to \infty} \hat{y}_{k_{j-1}}=\hat{y}.
\end{equation*}
By \eqref{(3.26)}, we get
\begin{equation*}
\left \|(p_{x}^{k_j},p_{y}^{k_j})  \right \| \le \varrho\left \|(x _{k_j}-\hat{x}_{k_{j-1}},y_{k_j}-\hat{y}_{k_{j-1}})  \right \|.
\end{equation*}
So
$$(p_{x}^{k_j},p_{y}^{k_j})\to (0,0) \ \ \ as \ \ \ j\to \infty .$$\\
Based on the result $\lim_{j \to \infty} L(x_{k_j},y_{k_j})=L(\hat{x} ,\hat{y})$, $(p_{x}^{k_j},p_{y}^{k_j})\in \partial L(x_{k_j},y_{k_j})$ and the closedness property of $\partial L$, we conclude that $(0,0)\in \partial L(\hat{x},\hat{y})$, which means $\hat{z}=(\hat{x},\hat{y})$ is a critical point of $L$, and $\mathcal{L}(z_0)\subset {\rm crit} \ L$.

{\rm (iii)} We prove the assertion by contradiction. Assume that $\lim_{k \to \infty} {\rm dist}\left ( z_k,\mathcal{L}(z_0) \right ) \ne 0$. Then, there exists a subsequence $\left \{  z _{k_m}\right \}$ and a constant $M>0$ such that
\begin{equation}
\label{(3.30)}
\left \| z_{k_m} -\hat{z} \right \| \ge {\rm dist}\left ( z _{k_{m}},\mathcal{L}(z_0) \right )>M,\ \ \ \forall \hat{z}\in \mathcal{L}(z_0).
\end{equation}
On the other hand, $\left \{  z _{k_m}\right \}$ is bounded and has a subsequence $\left \{  z _{k_{m_j}}\right \}$ converging to a point in $\mathcal{L}(z_0)$. Thus,
$$\lim_{j \to \infty} {\rm dist}\left ( z _{k_{m_j}},\mathcal{L}(z_0) \right ) = 0,$$
which is a contradiction to \eqref{(3.30)}.
\end{proof}
\par Now, we can prove the main convergence results of proposed algorithms under K{\L} property.
\begin{theorem}
\label{th31}
{\it  Suppose Assumptions \ref{Assumption31} and \ref{Assumption32} hold and $\left \{  z _{k}\right \}$ is a sequence generated by Algorithm \ref{alg1} with initial point $z_0$. Assume that $L$ is a K{\L} function. Then the following  results hold.

{\rm (i)} $\sum_{k=0}^{\infty  } \left \| z_{k+1}-z_k \right \|<\infty $,

{\rm (ii)} The sequence $\left \{  z _{k}\right \}$ converges to a critical point of $L$.
}
\end{theorem}
\begin{proof}
In the process of our proof, we always assume $L(z_{k})\ne L(\hat{z} )$. Otherwise, there exists an integer $\hat{k}$ such that $L(z_{\hat{k}})= L(\hat{z} )$. The sufficient decrease condition (Lemma \ref{lem31}) implies $z_{\hat{k}+1}= z_{\hat{k}}$. It follows that $z_{\hat{k}}=\hat{z}$ for any $k\ge\hat{k}$ and the assertions holds trivially.

{\rm (i)} Since $\left \{  L(z _{k})\right \}$ is a nonincreasing sequence and $\lim_{k \to \infty} L(z_{k})=L(\hat{z} )=L^\ast$ from \eqref{(3.29)}, for any $\eta>0$, there exists a positive integer $k_0$ such that
$$L(\hat{z} )<L(z_{k})<L(\hat{z} )+\eta, \ \ \ \forall k>k_0,$$
which means
$$z_k\in\left [L(\hat{z} )<L(z_{k})<L(\hat{z} )+\eta  \right ], \ \ \ \forall k>k_0.$$
On the other hand, $\lim_{k \to \infty} {\rm dist}\left ( z_k,\mathcal{L}(z_0) \right ) =0$. Therefore for any $\varepsilon >0$, there exists a positive integer $k_1$, such that
$${\rm dist}\left ( z_k,\mathcal{L}(z_0) \right )<\varepsilon, \ \ \ \forall k>k_1.$$
Let $l=\max\left \{ k_0,k_1 \right \}+1 $. Then for all $k\ge l$, we have
$$z_k\in \left \{ z\mid {\rm dist}\left ( z_k,\mathcal{L}(z_0) \right )<\varepsilon  \right \} \bigcap \left [L(\hat{z} )<L(z_{k})<L(\hat{z} )+\eta  \right ].$$
Note that $L$ is a constant on the compact set $\mathcal{L}(z_0)$. According to Lemma \ref{lem21}, there exists a concave function $\varphi  \in \Phi_\eta  $ such that
\begin{equation}
\label{(3.31)}
\varphi'\left ( L(z_{k})-L(\hat{z} ) \right )  {\rm dist}(0,\partial L(z_{k}))\geq 1, \ \ \ \forall k>l.
\end{equation}
From Lemma \ref{lem32}, we obtain that
\begin{equation}
\label{(3.32)}\
\aligned
{\rm dist}(0,\partial L(z_{k}))\le\left \|  (p_{x}^{k},p_{y}^{k}) \right \|
&\le \varrho\left \| \left (x_{k}-\hat{x}_{k-1},y_{k}-\hat{y}_{k-1}\right ) \right \|\\
&=\varrho\left \|z_{k}-\omega _{k-1}\right \|.
\endaligned
\end{equation}
Substituting \eqref{(3.32)} into \eqref{(3.31)}, we get
$$\varphi'(L(z_{k})-L(\hat{z} ))\ge \frac{1}{{\rm dist}(0,\partial L(z_{k}))} \ge \frac{1}{\varrho\left \|z_{k}-\omega _{k-1}\right \|}. $$
From the concavity of $\varphi$, we have
$$\varphi(L(z_{k+1})-L(\hat{z} ))\le \varphi(L(z_{k})-L(\hat{z} ))+\varphi'(L(z_{k})-L(\hat{z} ))(L(z_{k+1})-L(z_{k})).$$
It follows that
\begin{equation}
\label{(3.33)}\
\aligned
\varphi(L(z_{k})-L(\hat{z} ))-\varphi(L(z_{k+1})-L(\hat{z} ))
&\ge\varphi'(L(z_{k})-L(\hat{z} ))(L(z_{k})-L(z_{k+1}))\\
&\ge\frac{\rho\left \|z_{k+1}-\omega _{k}\right \|^2}{\varrho\left \|z_{k}-\omega _{k-1}\right \|}.
\endaligned
\end{equation}
The last inequality is from Lemma \ref{lem31}.
\par For convenience, we define $\triangle _k=\varphi(L(z_{k})-L(\hat{z} ))$. It is obvious that $\triangle _k$ is nonincreasing of $k$. Let $\underline{\triangle} =\inf_k\left \{ \triangle _k \right \} $ and $C=\frac{\varrho}{\rho} $. Then \eqref{(3.33)} can be simplified as
$$\triangle _k-\triangle _{k+1}\ge\frac{\left \|z_{k+1}-\omega _{k}\right \|^2}{C\left \|z_{k}-\omega _{k-1}\right \|},$$
i.e.,
$$\left \|z_{k+1}-\omega _{k}\right \|^2\le C(\triangle _k-\triangle _{k+1})\left \|z_{k}-\omega _{k-1}\right \|.$$
Using the fact that $2\sqrt{ab} \le a+b$ for $a,b\ge0$, we infer
\begin{equation}
\label{(3.34)}
2\left \|z_{k+1}-\omega _{k}\right \|\le C(\triangle _k-\triangle _{k+1})+\left \|z_{k}-\omega _{k-1}\right \|.
\end{equation}
Summing up \eqref{(3.34)} for $k=l+1,\dots ,K$ yields
\begin{equation*}
\aligned
2\sum_{k=l+1}^{K} \left \|z_{k+1}-\omega _{k}\right \|
&\le C(\triangle _{l+1}-\triangle _{K+1})+\sum_{k=l+1}^{K} \left \|z_{k}-\omega _{k-1}\right \|\\
&=C(\triangle _{l+1}-\triangle _{K+1})+\left \|z_{l+1}-\omega _{l}\right \|-\left \|z_{K+1}-\omega _{K}\right \|+\sum_{k=l+1}^{K} \left \|z_{k+1}-\omega _{k}\right \|.
\endaligned
\end{equation*}
Eliminating the same terms of the inequality, we have
\begin{equation*}
\aligned
\sum_{k=l+1}^{K} \left \|z_{k+1}-\omega _{k}\right \|
&\le C(\triangle _{l+1}-\triangle _{K+1})+\left \|z_{l+1}-\omega _{l}\right \|-\left \|z_{K+1}-\omega _{K}\right \|\\
&\le C(\triangle _{l+1}-\underline{\triangle})+\left \|z_{l+1}-\omega _{l}\right \|-\left \|z_{K+1}-\omega _{K}\right \|\\
&<\infty.
\endaligned
\end{equation*}
Let $K\to\infty$, we get
\begin{equation}
\label{(3.35)}
\sum_{k=l+1}^{K} \left \|z_{k+1}-\omega _{k}\right \|<\infty.
\end{equation}
Note that $(z_{k+1},\omega _k)=(x_{k+1}-\hat{x}_k,y_{k+1}-\hat{y}_k)$, and investigating the iterative point $\omega _k=(\hat{x}_k,\hat{y}_k)$ in Algorithm \ref{alg1}. If $(\hat{x}_k,\hat{y}_k)$ is generated by \eqref{(3.3)}, then
\begin{equation}
\label{(3.36)}\
\aligned
&(x_{k+1}-\hat{x}_k,y_{k+1}-\hat{y}_k)\\
=&(x_{k+1}-x_k-\alpha_k(x_{k+1}-x_{k})-\beta_k(x_k-x_{k-1}),y_{k+1}-y_k-\alpha_k(y_{k+1}-y_{k})-\beta_k(y_k-y_{k-1}))\\
=&z_{k+1}-z_k-\alpha_k(z_{k}-z_{k-1})-\beta_k(z_{k-1}-z_{k-2}).
\endaligned
\end{equation}
If $(\hat{x}_k,\hat{y}_k)$ is generated by \eqref{(3.4)}, then
$$(x_{k+1}-\hat{x}_k,y_{k+1}-\hat{y}_k)=(x_{k+1}-x_k,y_{k+1}-y_k)=z_{k+1}-z_k.$$
No matter how $(\hat{x}_k,\hat{y}_k)$ is generated, we always have
\begin{equation}
\label{(3.37)}\
\left \|z_{k+1}-\omega _k  \right \| 
\ge\left \|z_{k+1}-z_k  \right \|-\alpha_k\left \|z_{k}-z_{k-1} \right \|-\beta_k\left \|z_{k-1}-z_{k-2}\right \|.
\end{equation}
Summing up \eqref{(3.37)} for $k=l,\dots ,K$ we get
\begin{equation}
\label{(3.38)}\
\aligned
&\sum_{k=l+1}^{K} \left \|z_{k+1}-z_k\right \|-\sum_{k=l+1}^{K}\alpha_k\left \|z_{k}-z_{k-1} \right \|-\sum_{k=l+1}^{K} \beta_k\left \|z_{k-1}-z_{k-2}\right \|\\
\le&\sum_{k=l+1}^{K} \left \|z_{k+1}-\omega _k\right \|
< \infty .
\endaligned
\end{equation}
Note that $\alpha_k\in \left [ 0, \alpha _{\max}\right ],\beta _k\in \left [ 0, \beta _{\max}\right ]$. Let $\overline{\alpha } ={\sup}_k\left \{ \alpha_k \right \} ,\overline{\beta } ={\sup}_k\left \{ \beta_k \right \} $, then $0\le \overline{\alpha }+\overline{\beta } \le \alpha_{\max}+\beta_{\max}<1$, and it holds that
\begin{equation}
\label{(3.39)}\
\aligned
&\sum_{k=l+1}^{K} \left \|z_{k+1}-z_k\right \|-\overline{\alpha }\sum_{k=l+1}^{K}\left \|z_{k}-z_{k-1} \right \|-\overline{\beta }\sum_{k=l+1}^{K} \left \|z_{k-1}-z_{k-2}\right \|\\
\le&\sum_{k=l+1}^{K} \left \|z_{k+1}-z_k\right \|-\sum_{k=l+1}^{K}\alpha_k\left \|z_{k}-z_{k-1} \right \|-\sum_{k=l+1}^{K} \beta_k\left \|z_{k-1}-z_{k-2}\right \|.
\endaligned
\end{equation}
and
\begin{equation}
\label{(3.40)}\
\aligned
&\sum_{k=l+1}^{K} \left \|z_{k+1}-z_k\right \|-\overline{\alpha }\sum_{k=l+1}^{K}\left \|z_{k}-z_{k-1} \right \|-\overline{\beta }\sum_{k=l+1}^{K} \left \|z_{k-1}-z_{k-2}\right \|\\
=&\sum_{k=l+1}^{K} \left \|z_{k+1}-z_k\right \|-\overline{\alpha }\sum_{k=l+1}^{K}\left \|z_{k+1}-z_{k} \right \|-\overline{\alpha }(\left \|z_{l+1}-z_{l} \right \|-\left \|z_{K+1}-z_{K} \right \|)-\overline{\beta }\sum_{k=l+1}^{K} \left \|z_{k+1}-z_{k}\right \|\\
&-\overline{\beta }(\left \|z_{l+1}-z_{l} \right \|+\left \|z_{l}-z_{l-1} \right \|-\left \|z_{K+1}-z_{K} \right \|-\left \|z_{K}-z_{K-1} \right \|)\\
=&(1-\overline{\alpha }-\overline{\beta })\sum_{k=l+1}^{K} \left \|z_{k+1}-z_k\right \|-(\overline{\alpha }+\overline{\beta })(\left \|z_{l+1}-z_{l} \right \|-\left \|z_{K+1}-z_{K} \right \|)\\
&-\overline{\beta }(\left \|z_{l}-z_{l-1} \right \|-\left \|z_{K}-z_{K-1} \right \|).
\endaligned
\end{equation}
Combining \eqref{(3.38)}, \eqref{(3.39)} and \eqref{(3.40)}, we have
\begin{equation}
\label{(3.41)}\
\aligned
&(1-\overline{\alpha }-\overline{\beta })\sum_{k=l+1}^{K} \left \|z_{k+1}-z_k\right \|\\
=&(\overline{\alpha }+\overline{\beta })(\left \|z_{l+1}-z_{l} \right \|-\left \|z_{K+1}-z_{K} \right \|)+\overline{\beta }(\left \|z_{l}-z_{l-1} \right \|-\left \|z_{K}-z_{K-1} \right \|)+\sum_{k=l+1}^{K} \left \|z_{k+1}-z_k\right \|\\&-\overline{\alpha }\sum_{k=l+1}^{K}\left \|z_{k}-z_{k-1} \right \|-\overline{\beta }\sum_{k=l+1}^{K} \left \|z_{k-1}-z_{k-2}\right \|\\
\le&(\overline{\alpha }+\overline{\beta })(\left \|z_{l+1}-z_{l} \right \|-\left \|z_{K+1}-z_{K} \right \|)+\overline{\beta }(\left \|z_{l}-z_{l-1} \right \|-\left \|z_{K}-z_{K-1} \right \|)+\sum_{k=l+1}^{K} \left \|z_{k+1}-z_k\right \|\\&-\sum_{k=l+1}^{K}\alpha_k\left \|z_{k}-z_{k-1} \right \|-\sum_{k=l+1}^{K}\beta _k \left \|z_{k-1}-z_{k-2}\right \|\\
\le&(\overline{\alpha }+\overline{\beta })(\left \|z_{l+1}-z_{l} \right \|-\left \|z_{K+1}-z_{K} \right \|)+\overline{\beta }(\left \|z_{l}-z_{l-1} \right \|-\left \|z_{K}-z_{K-1} \right \|)+\sum_{k=l+1}^{K} \left \|z_{k+1}-\omega _k\right \|\\
<&\infty.
\endaligned
\end{equation}
The last inequality holds from \eqref{(3.35)}. Taking the limit as $K\to\infty$, and using the fact $\overline{\alpha }+\overline{\beta }<1$, we obtain
$$\sum_{k=l+1}^{\infty} \left \|z_{k+1}-z_k\right \|<\infty.$$
This shows that
\begin{equation}
\label{(3.42)}\
\sum_{k=0}^{\infty} \left \|z_{k+1}-z_k\right \|<\infty.
\end{equation}

{\rm (ii)}  \eqref{(3.42)} implies that for any $m>n\ge l$, it holds that
$$\left \| z_{m}-z_n \right \|=\left \|\sum_{k=n}^{m-1} (z_{k+1}-z_k)  \right \| \le\sum_{k=n}^{m-1}\left \| z_{k+1}-z_k \right \|< \sum_{k=n}^{\infty}\left \| z_{k+1}-z_k \right \|.$$
Taking $n\to\infty$, we obtain $\left \| z_{m}-z_n \right \|\to 0$ which means that $\left \{  z _{k}\right \}$ is a Cauchy sequence, and hence is a convergent sequence. We also know that $\left \{  z _{k}\right \}$ converges to a critical point of $L$ from Lemma \ref{lem34}(ii).
\end{proof}
Since K{\L} property is also a very useful tool to establish the convergence rate of many first-order methods. Based on K{\L} inequality, Attouch and Bolte \cite{ABB} first established convergence rate results which are related to the desingularizing function for proximal algorithms. Similar to the derivation process of \cite{ABB}, we can obtain convergence rate results as following.
\begin{theorem}
\label{th32}
{\it  (Convergence rate) Let Assumption \ref{Assumption31} and \ref{Assumption32} hold and let $\left \{  z _{k}\right \}$ be a sequence generated by Algorithm \ref{alg1} with $z_0=(x_0,y_0)$ as initial point. Assume also that $L$ is a K{\L} function and the desingularizing function has the form of $\varphi (t)=\frac{C}{\theta }t^{\theta }$ with $\theta \in (0,1]$, $C > 0$. Let $L^{\ast }=L(z)$, $\forall z\in\mathcal{L}(z_0)$. The following assertions hold.

{\rm (i)} If $\theta=1$, the Algorithm \ref{alg1} terminates in finite steps.

{\rm (ii)} If $\theta\in[\frac{1}{2} ,1)$, then there exist $\omega>0$ and $k_0\in \mathbb{N}$ such that
$$L(z_k)-L^{\ast }\le\mathcal{O} \left (  \exp\left ( -\frac{\omega}{\varrho }\right )\right ) , \ \ \ \forall k>k_0.$$

{\rm (iii)} If $\theta\in(0,\frac{1}{2})$, then there exist $\omega>0$ and $k_0\in \mathbb{N}$ such that
$$L(z_k)-L^{\ast }\le\mathcal{O} \left (\left ( \frac{k-k_0}{\varrho }\right ) ^{\frac{-1}{1-2\theta } }\right ), \ \ \ \forall k>k_0.$$
}
\end{theorem}
The result is almost the same as it mentioned in \cite{LL,LYV}. We omit the proof here.
\section{Numerical experiments}\label{sect4}
\hspace*{\parindent}In this subsection, we provide some numerical experiments which we carried out in order to illustrate the numerical results of Algorithm \ref{alg1} and \ref{alg2} with different Bregman distances.
The following list are various functions with its Bregman distances:\\
%
\hspace*{\parindent}(i) Define the function $\varphi _1(x)=-\gamma\sum_{i=1}^m \ln x_i$ with domain
$$\text{dom}\varphi _1=\{x=(x_1, x_2,\cdots, x_m)^T\in \mathbb{R}^m: x_i > 0, i =1, 2,\cdots, m\}$$
 and range ran$\varphi _1=(-\infty,+\infty)$. Then
$$\nabla \varphi _1(x)=\gamma(-\frac{1}{x_1}, -\frac{1}{x_2}, \cdots, -\frac{1}{x_m})^T$$
and the Bregman distance (the Itakura-Saito distance) with respect to $\varphi _1 $ is
$$D_{\varphi _1}(x, y) = \gamma\sum_{i=1}^m\big(\frac{x_i}{y_i}-\ln\big(\frac{x_i}{y_i}\big)-1\big),\ \ \forall x,y\in \mathbb{R}_{++}^m.$$
\hspace*{\parindent}(ii) Define the function $\varphi_2(x)=\frac{\gamma}{2}\|x\|^2$ with domain $\text{dom}\phi _2=\mathbb{R}^m$
 and range ran$\varphi_2=[0,+\infty)$. Then $\nabla \varphi_2(x)=x$ and the Bregman distance (the squared Euclidean distance) with respect to $\varphi_2$ is
$$D_{\varphi_2}(x, y) = \frac{\gamma}{2}\|x-y\|^2,\ \ \forall x,y\in \mathbb{R}^m.$$
It  is clear that $\varphi_i$ is $\gamma$-strongly convex ($i=1,2$).\\
\subsection{Nonconvex quadratic programming}
\ \par We consider the following quadratic programming problem\\
\begin{equation}
\aligned
\label{(4.1)}\
\min_{x}&\ \ \frac{1}{2} x^{T}Ax+b^{T}x \\
&s.t. \ \ x\in S,
\endaligned
\end{equation}
where $A$ is a symmetric matrix but not necessarily positive semidefinite, $b\in \mathbb{R}^n$ is a vector and $S \subset\mathbb{R}^n$ is a ball.
By introducing an auxiliary variable $y\in \mathbb{R}^n$, \eqref{(4.1)} can be reformulated as
\begin{equation}
\aligned
\label{(4.2)}\
\min_{x,\ y\in\mathbb{R}^{n}}\ \ &\frac{1}{2} y^{T}Ay+b^{T}y+\iota_S(x) \\
&s.t. \ \ x=y,
\endaligned
\end{equation}
where $\iota_S(x)$ is the indictor function with respect to ball $S$, defined by
$$\iota_S(x)=\left\{\begin{matrix}
  0, & x\in S, \\
  +\infty, &x\notin S.
\end{matrix}\right.$$

We use penalty method to handle the constraint. The problem can be transformed into
\begin{equation}
\label{(4.3)}\
\min_{x,\ y\in\mathbb{R}^{n}}\ \ \frac{1}{2} y^{T}Ay+b^{T}y+\iota_S(x)+\frac{\mu}{2}\left \| x-y \right \|^{2},
\end{equation}
where $\mu>0$ be a penalty parameter. When $\mu$ is large enough, the solution of \eqref{(4.3)} is an approximate solution of problem \eqref{(4.2)}.
\par Let
$$
\aligned
&f(x)\equiv0, \ g(y)=\frac{1}{2} y^{T}Ay+b^{T}y,\\
&Q(x,y)=\iota_S(x)+\frac{\mu}{2}\left \| x-y \right \|^{2}.
\endaligned
$$
Obviously, $f$ and $g$ are smooth functions, and the Lipschitz constant of $\nabla g$, i.e. $L_{\nabla g}$, is the maximal singular value of $A$. Note that when $A$ is not positive semidefinite, $g$ is a nonconvex function. We can solve \eqref{(4.3)} by Algorithm \ref{alg1} and Algorithm \ref{alg2}.
Let $\phi_2(y)=\frac{\lambda }{2} \left \| y\right \|^{2}$, we can elaborate the $x$-subproblem and $y$-subproblem of our algorithms respectively as follows.
\par The $x$-subproblem corresponds to the following optimization problem
$$
\aligned
x_{k+1}
&\in\arg\min_{ x\in \mathbb{R}^n}\left \{ Q(x,\hat{y}_k )+\left \langle \nabla f(\hat{x}_k),x \right \rangle+D_{\phi_1}(x,\hat{x}_k)  \right \} \\
&=\arg\min_{ x\in \mathbb{R}^n}\left \{ \iota_S(x)+\frac{\mu}{2}\left \| x-\hat{y}_k \right \|^{2}+D_{\phi_1}(x,\hat{x}_k)  \right \} \\
&=\arg\min_{ x\in S}\left \{ \frac{\mu}{2}\left \| x-\hat{y}_k \right \|^{2}+D_{\phi_1}(x,\hat{x}_k)  \right \}. \\
\endaligned
$$
The $y$-subproblem corresponds to the following optimization problem
$$
\aligned
y_{k+1}
&\in\arg\min_{ y\in \mathbb{R}^n}\left \{ Q(x_{k+1},y )+\left \langle \nabla g(\hat{y}_k),y \right \rangle+D_{\phi_2}(y,\hat{y}_k) \right \} \\
&=\arg\min_{ y\in \mathbb{R}^n}\left \{ \frac{\mu}{2}\left \|x_{k+1}-y \right \|^{2}+\left \langle  A\hat{y}_k+b,y \right \rangle+\frac{\lambda }{2} \left \| y-\hat y_{k}  \right \|^{2}   \right \}, \\
\endaligned
$$
which has an explicit expression
$$y_{k+1} =\frac{1}{\mu+\lambda } \left (\mu x_{k+1}+\lambda\hat{y}_k - A\hat{y}_k-b \right) .$$
\par In numerical experiments, we set $A=D+D^T\in \mathbb{R}^{n\times n}$, where $D$ is a matrix generated by i.i.d. standard Gaussian entries. The vector $b$ is also generated by i.i.d. standard Gaussian entries. We take $n=500$ and the radius of the ball is $r=2$.
Since $f\equiv 0$, any positive number can be the Lipschitz constant of $L_{\nabla f}$. We set $L_{\nabla f}=L_{\nabla g}$. We selected the starting point randomly, and use
$$E_k=\|x_{k+1}-x_k\|+\|y_{k+1}-y_k\|<10^{-4}$$
as the stopping criteria. In the numerical results, ``Iter." denotes  the number of iterations. ``Time" denotes  the CPU time. ``Extrapolation" records the number of taking extrapolation step, i.e., the number of adopting \eqref{(3.3)}. 


In order to show the effectiveness of the proposed algorithms, we compare Algorithm \ref{alg1}, Algorithm \ref{alg2} with ASAP \cite{NT} and aASAP \cite{XX} for different Bregman distance. Note that when $\alpha_k\equiv \beta_k\equiv 0$, Algorithm \ref{alg1} and Algorithm \ref{alg2} correspond to ASAP. For aASAP, we take $\alpha_k=0.3$. For Algorithm \ref{alg1}, we set $\alpha_k=0.3, \beta_k=0.2$. And we also take extrapolation parameter dynamically updating with $\alpha_k=\beta_k=\frac{k-1}{k+2}$. Even if the theoretical bound of extrapolation parameter $\alpha_k+\beta_k$ with dynamically updating dost not permit to go beyond $1$, for the convergence is also obtained for this case with a better performance. For Algorithm \ref{alg2}, we set $\alpha_0=0.3, \beta_0=0.2$ as the initial extrapolation parameter and $t=1.2, \beta_{\max}=0.499$. We use ``Alg. 1-i" and ``Alg. 2-i" to denote Algorithm \ref{alg1} and Algorithm \ref{alg2} with $\phi_1 (x)=\varphi_i(x)(1\le i\le 2)$, respectively, where extrapolation parameter $\alpha_k=0.3, \beta_k=0.2$. We use ``Alg. 1-i(F)" and ``Alg. 2-i(F)" to denote Algorithm \ref{alg1} and Algorithm \ref{alg2} with $\phi_1 (x)=\varphi_i(x)(1\le i\le 2)$, respectively, where extrapolation parameter $\alpha_k=\beta_k=\frac{k-1}{k+2}$.
\par In Table \ref{table1}, we list the iterations, CPU time and extrapolation step of the above algorithm for different Bregman distance. In Figure \ref{fig_sim}, (a) and (b) reports the result of different extrapolation parameter, respectively, (c) reports the result of different Bregman distance. It can be seen that the Itakura-Saito distance have computational advantage than the squared Euclidean distance for Algorithm \ref{alg1} and Algorithm \ref{alg2} in terms of number of iteration and CPU time. Compared with one-step extrapolation and original algorithm, two-step extrapolation performs much better. It shows that Algorithm \ref{alg2} with adaptive extrapolation parameters performs the best among all algorithms. 

\begin{table}[!ht]
\centering
\caption{Numerical results of different Bregman distance with different extrapolation parameter}\label{table1}
\begin{tabular}{lllllllll}
\hline
           \multicolumn{4}{c}{Itakura-Saito distance} &  & \multicolumn{4}{c}{squared Euclidean distance} \\ \cline{1-4} \cline{6-9}
Algorithm      & Iter.  & Time(s)  & Extrapolation &                 &Algorithm    & Iter.  & Time(s)  & Extrapolation \\ \hline
ASAP        & 192    &0.5906         & \multicolumn{1}{c}{191}   &&ASAP  & 202    &0.6456         &\multicolumn{1}{c}{201}    \\
aASAP       & 138    &0.4127         & \multicolumn{1}{c}{137}   &&aASAP & 147    &0.4590         &\multicolumn{1}{c}{146}    \\
Alg. 1-1     & 81     &0.3169         & \multicolumn{1}{c}{72}    &&Alg. 1-2  & 98     &0.2725          &\multicolumn{1}{c}{96}    \\
Alg. 2-1     & 28     &0.0156         & \multicolumn{1}{c}{26}    &&Alg. 2-2& 33     &0.0469          &\multicolumn{1}{c}{28}    \\
Alg. 1-1(F)  & 44&0.0898& \multicolumn{1}{c}{31}                  &&Alg. 1-2(F) & 48     &0.1013          &\multicolumn{1}{c}{34}    \\ \hline
\end{tabular}
\end{table}
\begin{figure*}[!t]
    \centering
    \subfloat[Itakura-Saito distance]{\includegraphics[width=3.0in]{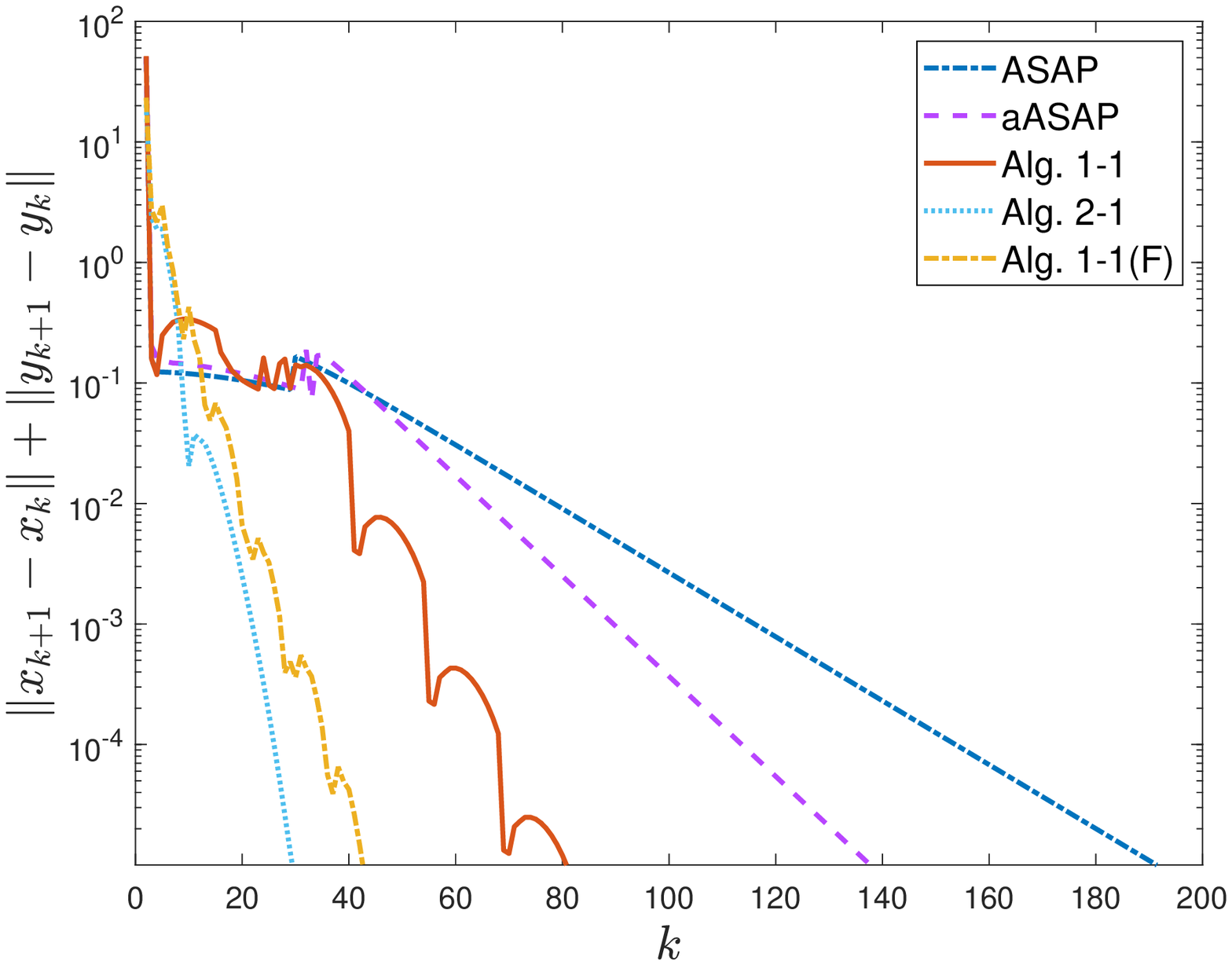}
    \label{fig_first_case}}
\hfil
\subfloat[squared Euclidean distance]{\includegraphics[width=3.0in]{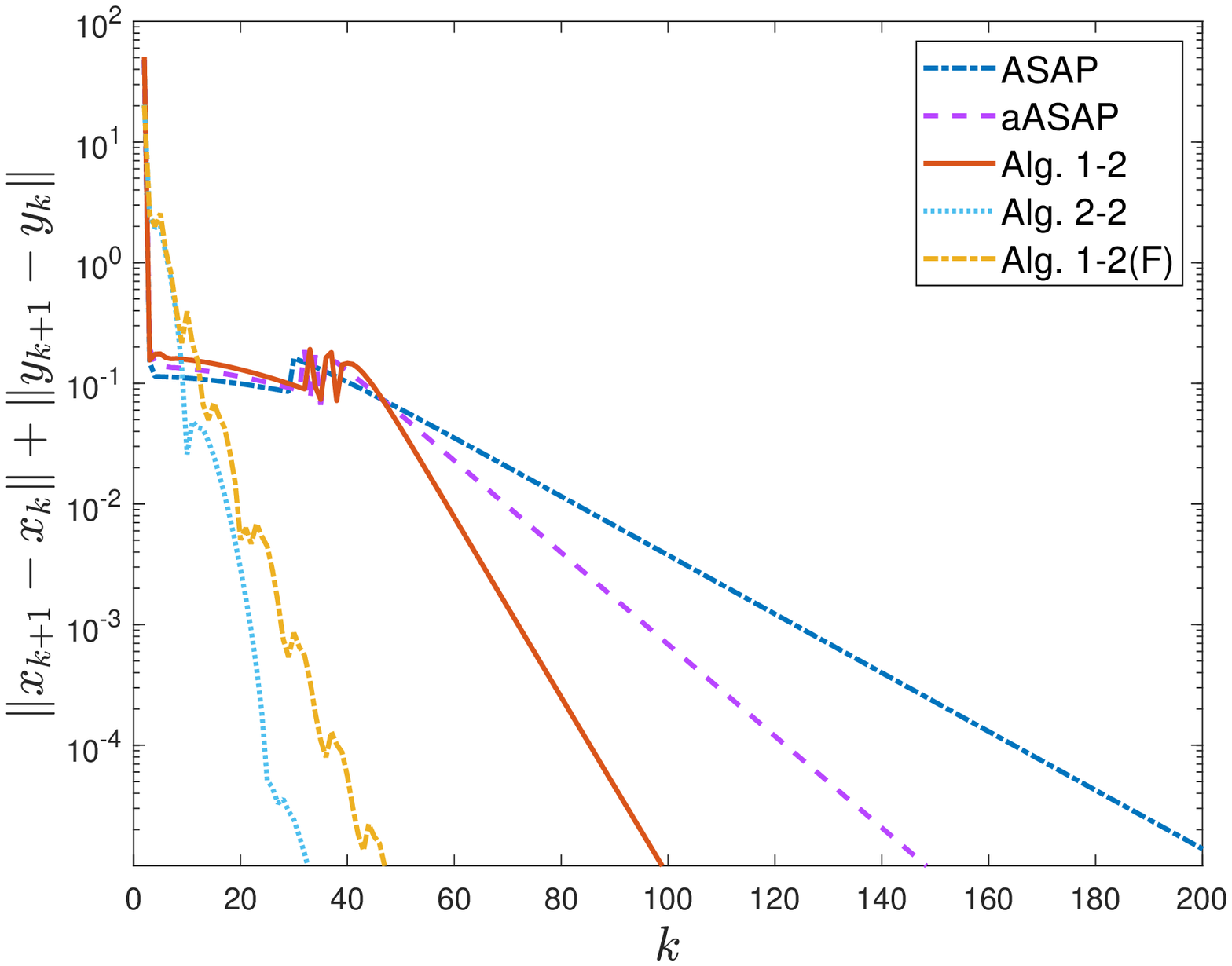}%
    \label{fig_second_case}}
\hfil
\subfloat[Itakura-Saito and squared Euclidean distance]{\includegraphics[width=3.0in]{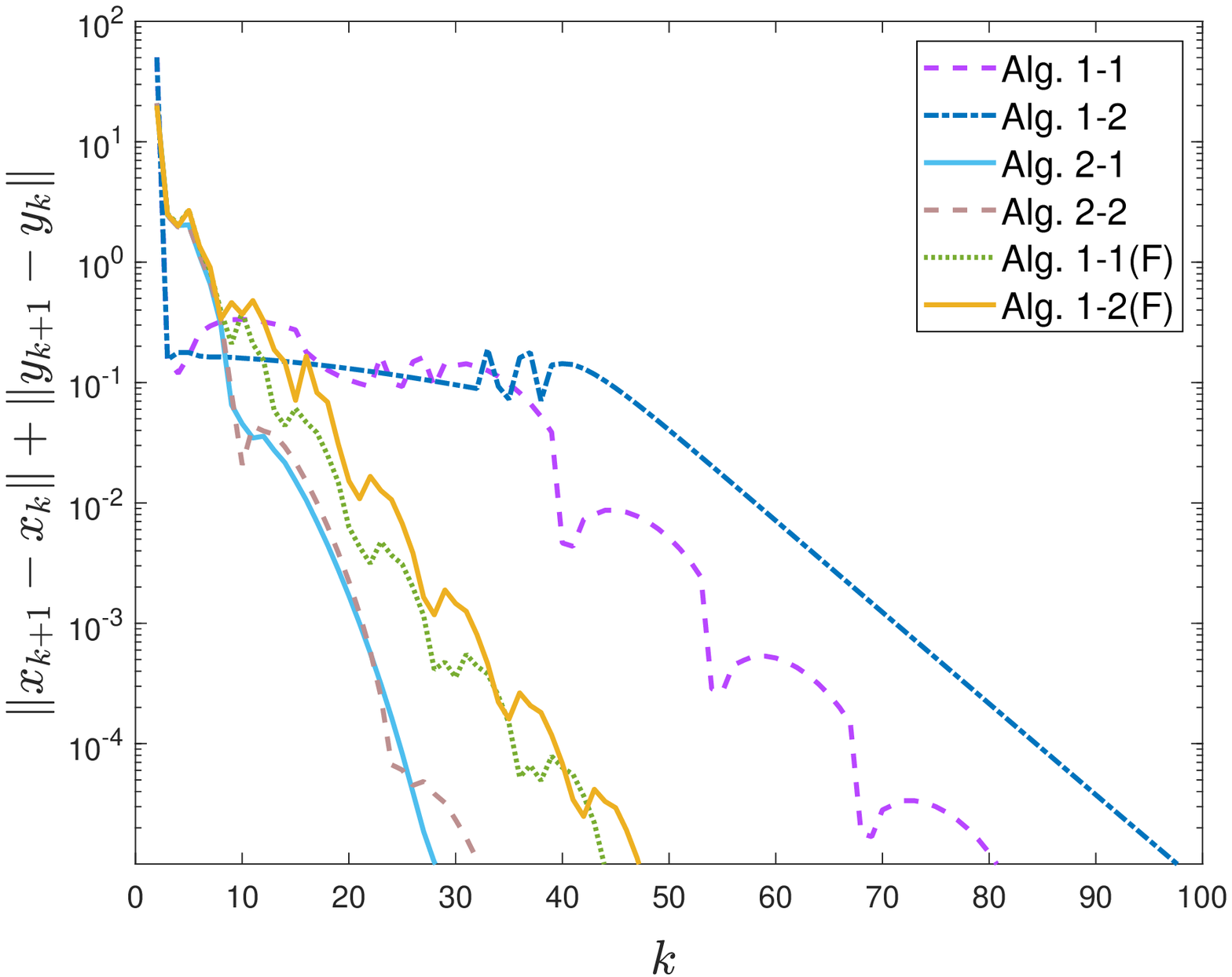}%
    \label{fig_thrid_case}}
    \caption{The value of $\|x_{k+1}-x_k\|+\|y_{k+1}-y_k\|$ versus the iteration numbers for different Bregman distance with different extrapolation parameter.}
    \label{fig_sim}
\end{figure*}

\subsection{Sparse logistic regression}
\ \par In this subsection, we apply our algorithms to solve the sparse logistic regression problem. It is an attractive extension to logistic regression as it can reduce overfitting and perform feature selection simultaneously. We consider the Capped-$l_1$ regularized logistic regression problem \cite{Z}, defined as
\begin{equation}
\label{(4.4)}\
\underset{x\in \mathbb{R}^d}{\min}\ \  \frac{1}{n} \sum_{i=1}^{n} \log(1+\exp(-b_ia_{i}^{T} x))+\lambda\sum_{j=1}^{d}\min(\left | x_j \right |,\theta  ),\ \  \theta >0, 
\end{equation}
where $x\in \mathbb{R}^d$, $a_i\in \mathbb{R}^d$, $b_i\in \left \{ -1,1 \right \}, i=1,\dots,n.$
\par By the similar method as the former example, we transform \eqref{(4.4)} to
\begin{equation}
\aligned
\label{(4.5)}\
\underset{x,y\in \mathbb{R}^d}{\min}\ \  &\frac{1}{n} \sum_{i=1}^{n} \log(1+\exp(-b_ia_{i}^{T} x))+\lambda\sum_{j=1}^{d}\min(\left | y_j \right |,\theta  )\\
&s.t. \ \ x=y,
\endaligned
\end{equation}
which can be described as
\begin{equation}
\label{(4.6)}\
\underset{x,y\in \mathbb{R}^d}{\min}\ \  \frac{1}{n} \sum_{i=1}^{n} \log(1+\exp(-b_ia_{i}^{T} x))+\lambda\sum_{j=1}^{d}\min(\left | y_j \right |,\theta  )+\frac{\mu}{2}\left \| x-y  \right \|^2.
\end{equation}
\par Let
\begin{equation}
\aligned
\label{(4.7)}\
&f(x)=\frac{1}{n} \sum_{i=1}^{n} \log(1+\exp(-b_ia_{i}^{T} x)),\ g(y)\equiv0,\\
&Q(x,y)=\lambda\sum_{j=1}^{d}\min(\left | y_j \right |,\theta  )+\frac{\mu}{2}\left \| x-y  \right \|^2.
\endaligned
\end{equation}
\par It is easy to verify that all the functions satisfy the assumptions. However, when $n$ is large, $L_{\bigtriangledown f} $ is difficult to compute, we cannot determine the strong convexity modulus range of the Bregman function. So we need to use backtracking strategy to solve the problem. Let $\phi_2(y)=\frac{\eta }{2} \left \| y\right \|^{2}$, we can elaborate the $x$-subproblem and $y$-subproblem of our algorithms respectively as follows.
\par The $x$-subproblem in \eqref{(4.6)} corresponds to the following optimization problem
$$
\aligned
x_{k+1}
&\in\arg\min_{ x\in \mathbb{R}^d}\left \{ Q(x,\hat{y}_k )+\left \langle \nabla f(\hat{x}_k),x \right \rangle+D_{\phi_1}(x,\hat{x}_k)  \right \} \\
&=\arg\min_{ x\in \mathbb{R}^d}\left \{ \frac{\mu}{2}\left \| x-\hat{y}_k \right \|^{2}+\left \langle \nabla f(\hat{x}_k),x \right \rangle+D_{\phi_1}(x,\hat{x}_k)  \right \}. \\
\endaligned
$$
The $y$-subproblem in \eqref{(4.6)} corresponds to
$$
\aligned
y_{k+1}
&\in\arg\min_{ y\in \mathbb{R}^d}\left \{ Q(x_{k+1},y)+\left \langle \nabla g(\hat{y}_k),y \right \rangle+D_{\phi_2}(y,\hat{y}_k)  \right \} \\
&=\arg\min_{ y\in \mathbb{R}^d}\left \{ \frac{\mu}{2}\left \| x_{k+1}-y \right \|^{2}+\lambda\sum_{j=1}^{d}\min(\left | y_j \right |,\theta  )+ \frac{\eta}{2}\left \| y-\hat{y}_k \right \|^{2} \right \} \\
&=\arg\min_{ y\in \mathbb{R}^d}\left \{ \lambda\sum_{j=1}^{d}\min(\left | y_j \right |,\theta  )+ \frac{\mu+\eta}{2}\left \| y-\frac{1}{\mu+\eta}(\eta \hat{y}_k+\mu x_{k+1}) \right \|^{2} \right \} \\
&=\arg\min_{ y\in \mathbb{R}^d}\left \{  \frac{\lambda}{\mu+\eta}\sum_{j=1}^{d}\min(\left | y_j \right |,\theta  )+ \frac{1}{2}\left \| y-\frac{1}{\mu+\eta}(\eta \hat{y}_k+\mu x_{k+1}) \right \|^{2} \right \}.
\endaligned
$$
\par For convenience, we set
\begin{equation}
\label{(4.8)}\
h(t,y,u_k)=\frac{\lambda}{t}\sum_{j=1}^{d}\min(\left | y_j \right |,\theta  )+ \frac{1}{2}\left \| y-u_k \right \|^{2}, 
\end{equation}
where $t=\mu+\eta$, $u_k=\frac{1}{\mu+\eta}(\eta \hat{y}_k+\mu x_{k+1}).$
Then $h(t,y,u_k)$ has a separable structure. It can be decomposed into $d$ one-dimensional subproblems. Let $\left ( u_{k} \right )_j$ be the $j$-th component of $u_k$. Then each component of $y_{k+1}$ can be calculated from
\begin{equation}
\label{(4.9)}\
\left ( y_{k+1} \right )_j=\arg\min_{ y_j\in \mathbb{R}}\left \{\frac{1}{2}\left( y_j-\left ( u_{k} \right )_j\right  )^{2}+\frac{\lambda}{t}\min(\left | y_j \right |,\theta  )  \right \},\ j=1,\dots, d,
\end{equation}
where $\left ( y_{k+1} \right )_j$ denotes the $j$-th component of $y_{k+1}$.
Then \eqref{(4.9)} can be described as follows
\begin{equation*}
\left\{
\begin{array}{r@{\extracolsep{0.5ex}}r@{}r@{}r@{}r@{}r@{}r@{}}
 &\tilde{y} _j=\underset{y_j\in \mathbb{R}}{\arg\min}\ \frac{1}{2}\left( y_j-\left ( u_{k} \right )_j\right  )^{2}+\frac{\lambda}{t}\theta  \ \ \  \ \ &\text{ if } \ \left | y_j \right |>\theta,\\
&\hat{y} _j=\underset{y_j\in \mathbb{R}}{\arg\min}\ \frac{1}{2}\left( y_j-\left ( u_{k} \right )_j\right  )^{2}+\frac{\lambda}{t}\left | y_j \right |  \ \ &\text{ if } \ \left | y_j \right |<\theta.              
\end{array}
\right.
\end{equation*}
Furthermore,
\begin{equation*}
\left\{
\begin{array}{r@{\extracolsep{0.5ex}}l@{}l@{}l@{}l@{}l@{}l@{}}
 &\tilde{y} _j&=&\text{sign}(\left( u_{k} \right)_j)\max(\theta , | \left( u_{k} \right)_j | )  \ \ \  \ \ &\text{ if } \ \left | y_j \right |>\theta,\\
&\hat{y} _j&=&\text{sign}(\left( u_{k} \right)_j)\min(\theta ,\max(0, | \left ( u_{k} \right )_j | -\frac{\lambda}{t}))  \ \ &\text{ if } \ \left | y_j \right |<\theta.              
\end{array}
\right.
\end{equation*}
Set $h_j(t,y_j,\left( u_{k} \right)_j)=\frac{\lambda}{t}\min(\left | y_j \right |,\theta  )+ \frac{1}{2}\left \| y_j-\left ( u_{k} \right )_j \right \|^{2}$. Then $\left ( y_{k+1} \right )_j$ can be expressed by
$$\left ( y_{k+1} \right )_j=\begin{cases}
  \tilde{y} _j& \text{ if } h_j(t,\tilde{y} _j,\left ( u_{k} \right )_j)\le h_j(t,\hat{y} _j,\left( u_{k} \right)_j), \\
  \hat{y}_j& \text{ otherwise.}
\end{cases}$$
\par In the experiment, we take $n = 500, d=200$ and the parameters of the problem are set as $\lambda =10^{-3},\theta =0.1\lambda $, which is the same as \cite{GZL}. The backtracking parameters of the algorithm are set as $\rho =2,\delta =10^{-5}$. 
We selected the starting point for all algorithms randomly and use
$$E_k=\|x_{k+1}-x_k\|+\|y_{k+1}-y_k\|<10^{-5}$$
as the stopping criteria. In the numerical results, ``Iter." denotes  the number of iterations. ``Time" denotes  the CPU time. ``Extrapolation" records the number of taking extrapolation step.

In order to show the effectiveness of the proposed algorithms, we compare our algorithms with ASAP \cite{NT} and aASAP \cite{XX}. The extrapolation parameter of aASAP is $\alpha_k=0.3$. In Algorithm \ref{alg1}, we set $\alpha_k=0.3, \beta_k=0.2$. And we also take extrapolation parameter dynamically updating with $\alpha_k=\beta_k=\frac{k-1}{k+2}$. For Algorithm \ref{alg2}, we set $\alpha_0=0.3, \beta_0=0.2$ as the initial extrapolation parameter and $t=1.5, \beta_{\max}=0.499$. We also use ``Alg. 1-i" and ``Alg. 2-i" to denote Algorithm \ref{alg1} and Algorithm \ref{alg2} with $\phi_1 (x)=\varphi_i(x)(1\le i\le 2)$, respectively, where extrapolation parameter $\alpha_k=0.3, \beta_k=0.2$. We use ``Alg. 1-i(F)" and ``Alg. 2-i(F)" to denote Algorithm \ref{alg1} and Algorithm \ref{alg2} with $\phi_1 (x)=\varphi_i(x)(1\le i\le 2)$, respectively, where extrapolation parameter $\alpha_k=\beta_k=\frac{k-1}{k+2}$.

Figure \ref{fig_sim2} shows the performance of different algorithms. We also list the iteration number, the CPU time, and the extrapolation number on the test set of each algorithm in Table \ref{table2}. Figure \ref{fig_sim2} (a) and (b) reports the result of different extrapolation parameter, respectively. It shows that Algorithm \ref{alg2} with adaptive extrapolation parameters performs the best among all algorithms. Figure \ref{fig_sim2} (c) reports the result of different Bregman distance. It shows that the Itakura-Saito distance have computational advantage than the squared Euclidean distance.

We also use BB rule improves the computational efficiency of each algorithm. The symbol in Table \ref{table3} and Figure \ref{fig_sim3} “$\sim$BB” means that the algorithm adopts the BB rule with lower bound of $t_{\min}=1.3$ in backtracking process. It shows that using BB rule to initialize the stepsize can improve the efficiency of all algorithms.

\begin{table}[!ht]
\centering
\caption{Numerical results of different Bregman distance with different extrapolation parameter}\label{table2}
\begin{tabular}{lllllllll}
\hline
           \multicolumn{4}{c}{Itakura-Saito distance} &  & \multicolumn{4}{c}{squared Euclidean distance} \\ \cline{1-4} \cline{6-9}
Algorithm & Iter.  & Time(s)  & Extrapolation &  &Algorithm      & Iter.  & Time(s)  & Extrapolation \\ \hline
ASAP       & 39     &9.3456         & \multicolumn{1}{c}{38}    &  &ASAP& 71     &18.3569         &\multicolumn{1}{c}{70}    \\
aASAP     & 35     &8.3227         & \multicolumn{1}{c}{34}    &  &aASAP& 56     &14.0903         &\multicolumn{1}{c}{55}    \\
Alg. 1-1   & 29     &6.9867         & \multicolumn{1}{c}{21}    &  &Alg. 1-2& 43     &11.0023          &\multicolumn{1}{c}{39}    \\
Alg. 2-1   & 15     &3.4024         & \multicolumn{1}{c}{7}     &  &Alg. 2-2& 23     &6.1576          &\multicolumn{1}{c}{16}    \\
Alg. 1-1(F) & 19&5.5435& \multicolumn{1}{c}{8}     &  &            Alg. 1-2(F)&      33     &7.9012          &\multicolumn{1}{c}{20}    \\ \hline
\end{tabular}
\end{table}

\begin{figure*}[!t]
    \centering
    \subfloat[Itakura-Saito distance]{\includegraphics[width=3.0in]{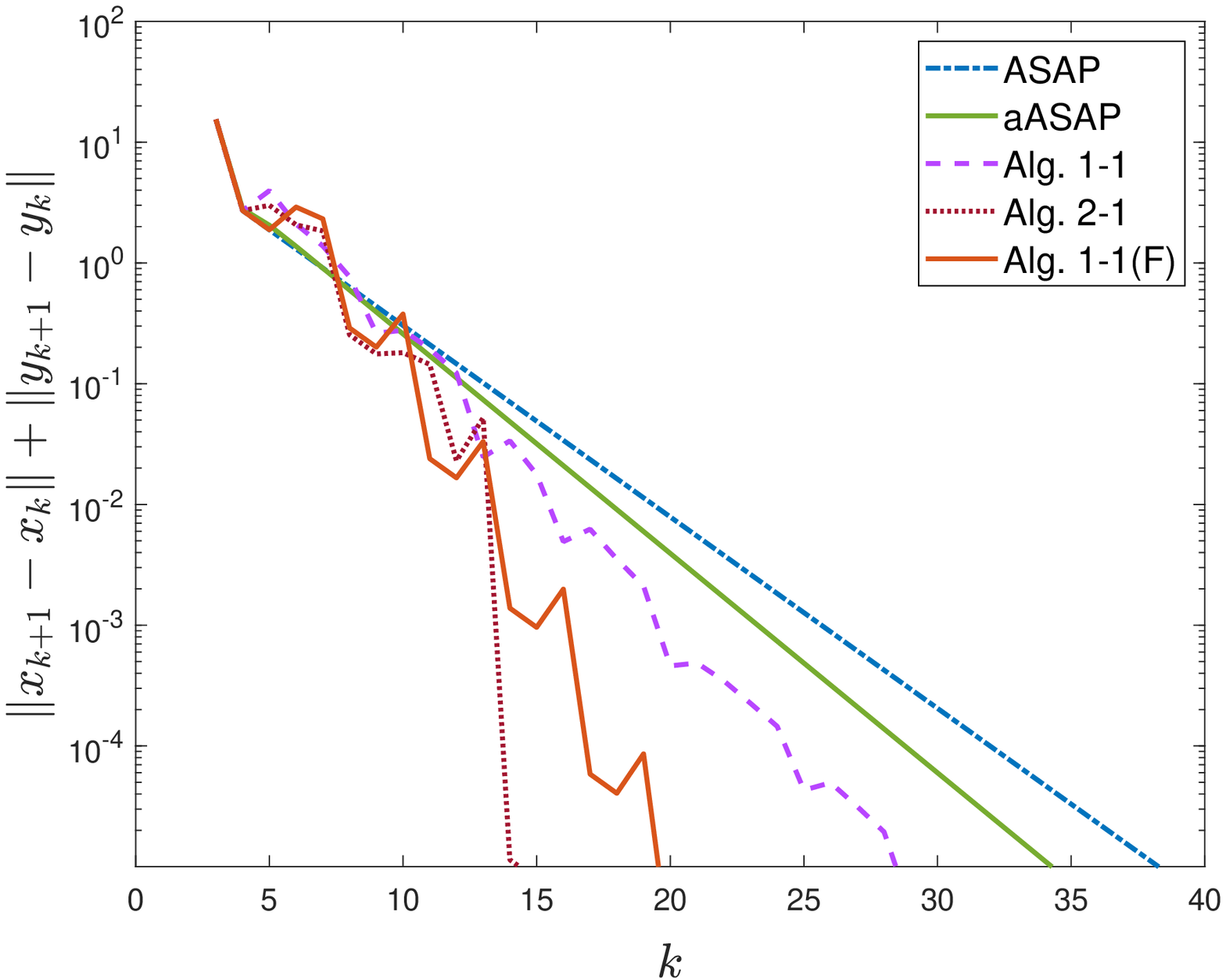}
    \label{fig_first_case}}
\hfil
\subfloat[squared Euclidean distance]{\includegraphics[width=3.0in]{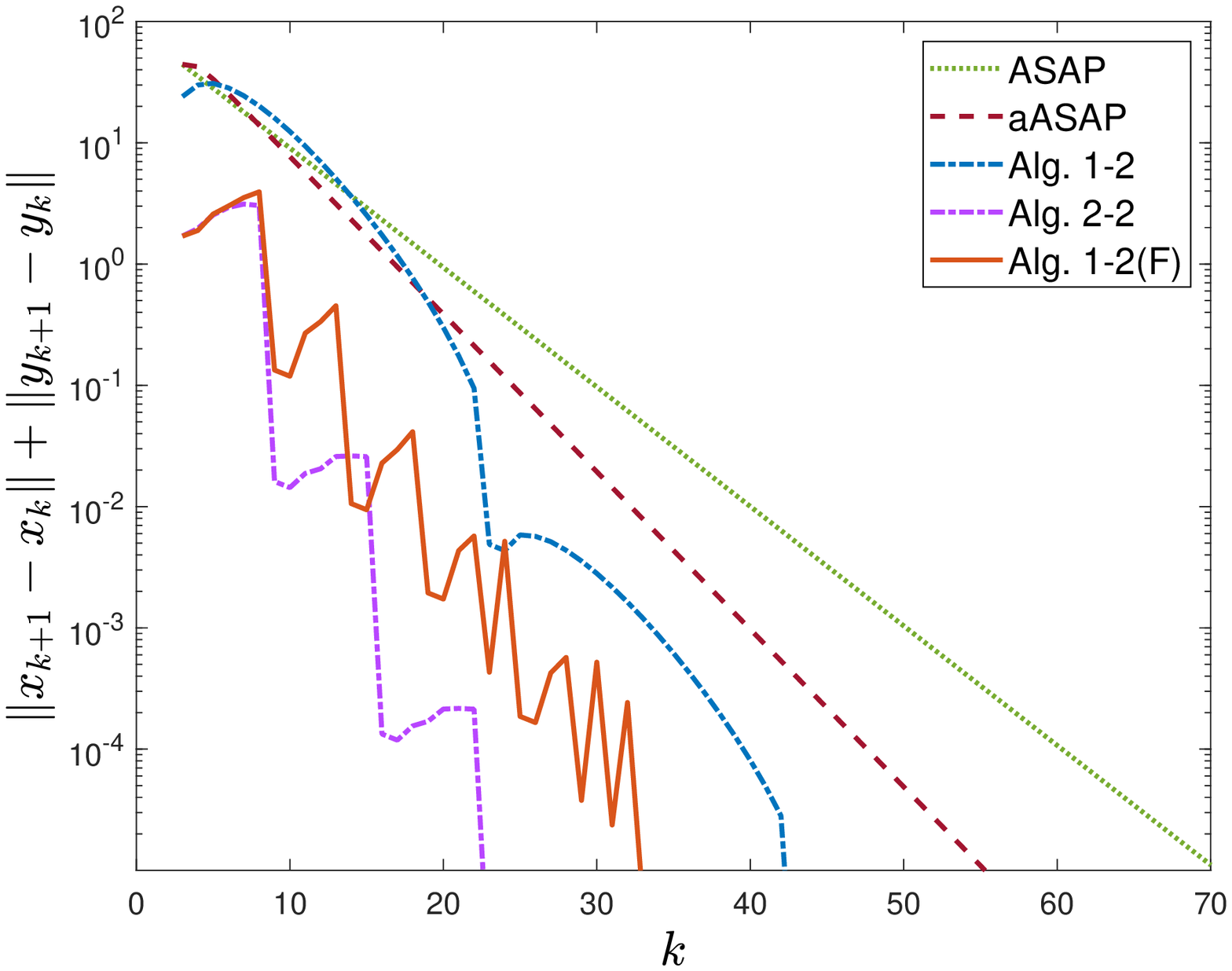}%
    \label{fig_second_case}}
\hfil
\subfloat[Itakura-Saito and squared Euclidean distance]{\includegraphics[width=3.0in]{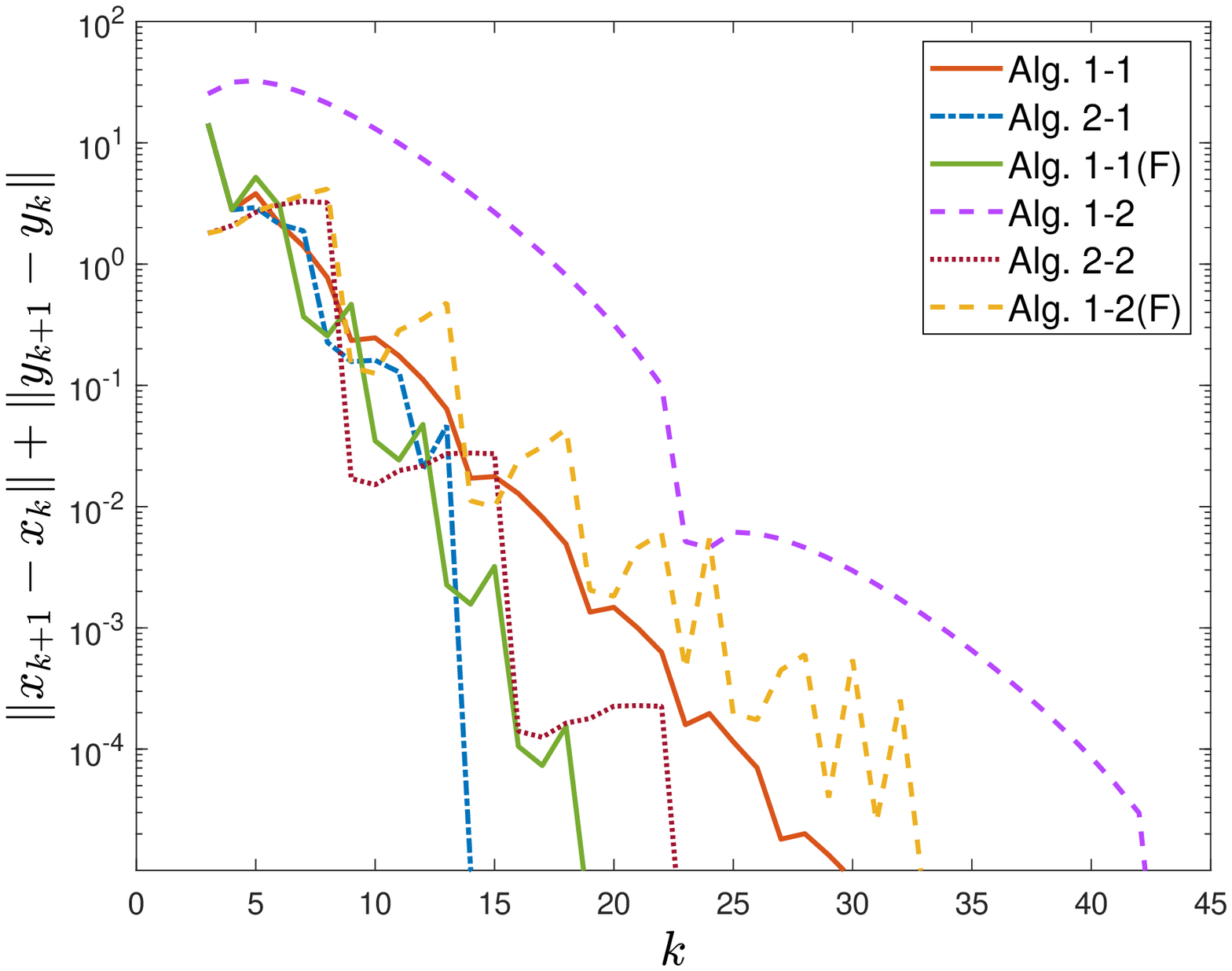}%
    \label{fig_thrid_case}}
    \caption{The value of $\|x_{k+1}-x_k\|+\|y_{k+1}-y_k\|$ versus the iteration numbers for Algorithm 1 with different extrapolation parameter.}
    \label{fig_sim2}
\end{figure*}

\begin{table}[!ht]
\centering
\caption{Numerical results of different Bregman distance with different extrapolation parameter}\label{table3}
\begin{tabular}{lllllllll}
\hline
           \multicolumn{4}{c}{Itakura-Saito distance} &  & \multicolumn{4}{c}{squared Euclidean distance} \\ \cline{1-4} \cline{6-9}
Algorithm & Iter.  & Time(s)  & Extra. &  &Algorithm      & Iter.  & Time(s)  & Extra. \\ \hline
ASAP$\sim$BB       & 15    &4.3476         & \multicolumn{1}{c}{14}    &  &ASAP$\sim$BB& 25     &6.4702         &\multicolumn{1}{c}{24}    \\
aASAP$\sim$BB     & 14     &4.0324         & \multicolumn{1}{c}{13}    &  &aASAP$\sim$BB& 21     &6.0604         &\multicolumn{1}{c}{20}    \\
Alg. 1-1$\sim$BB   & 13     &3.9198         & \multicolumn{1}{c}{7}    &  &Alg. 1-2$\sim$BB& 19     &5.3523          &\multicolumn{1}{c}{10}    \\
Alg. 2-1$\sim$BB   & 9     &2.9745         & \multicolumn{1}{c}{5}     &  &Alg. 2-2$\sim$BB& 15     &4.1270          &\multicolumn{1}{c}{8}    \\
Alg. 1-1(F)$\sim$BB & 11&3.5435& \multicolumn{1}{c}{6}     &  &            Alg. 1-2(F)$\sim$BB&17     &4.9352          &\multicolumn{1}{c}{9}    \\ \hline
\end{tabular}
\end{table}

\begin{figure*}[!t]
    \centering
    \subfloat[Itakura-Saito distance]{\includegraphics[width=3.0in]{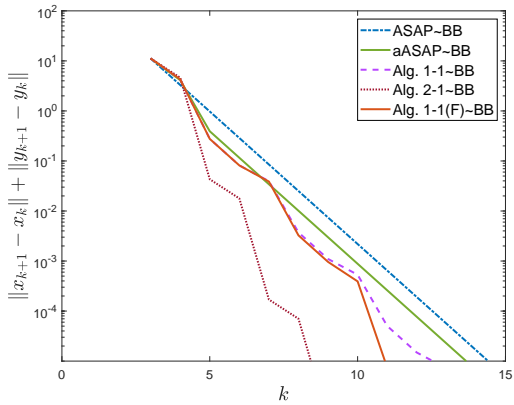}
    \label{fig_first_case}}
\hfil
\subfloat[squared Euclidean distance]{\includegraphics[width=3.0in]{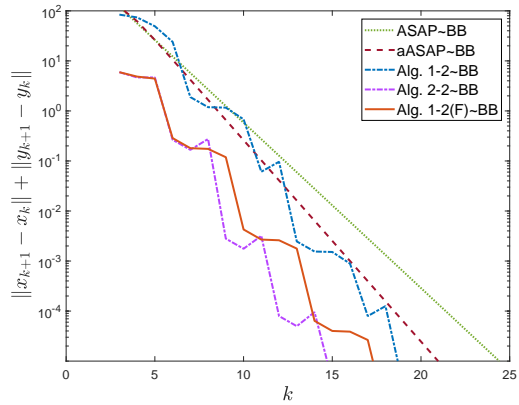}%
    \label{fig_second_case}}
\hfil
\subfloat[Itakura-Saito and squared Euclidean distance]{\includegraphics[width=3.0in]{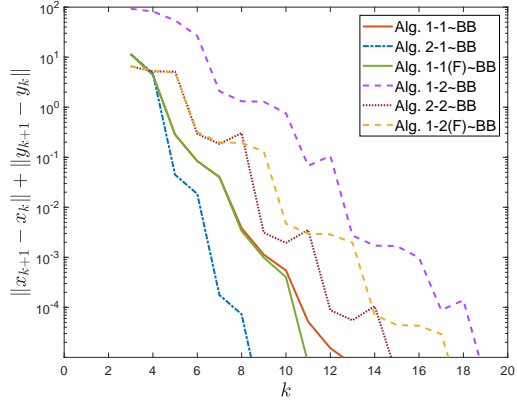}%
    \label{fig_thrid_case}}
    \caption{The value of $\|x_{k+1}-x_k\|+\|y_{k+1}-y_k\|$ versus the iteration numbers for Algorithm 1 with different extrapolation parameter.}
    \label{fig_sim3}
\end{figure*}
\section{Conclusion}
{In this paper, we introduce a two-step inertial Bregman alternating structure-adapted proximal gradient descent algorithm for solve nonconvex and nonsmooth nonseparable optimization problem.
Under some assumptions, we proved that our algorithm is a descent method in sense of objective function values, and every cluster point is a critical point of the objective function. The convergence of proposed algorithms are proved if the objective function satisfies the Kurdyka--{\L}ojasiewicz property. 
Furthermore, if the desingularizing function has the special form, we also established the linear and sub-linear convergence rates of the function value sequence generated by the algorithm. In addition, we also proposed a backtracking strategy with BB method to make
our algorithms more flexible when the Lipschitz constant is unknown or difficult to compute. In numerical experiments, 
we apply different Bregman distance to solve nonconvex quadratic programming and sparse logistic regression problem. Numerical results are reported to show the effectiveness of the proposed algorithm.}

\end{document}